\documentclass[11pt, notitlepage]{article}   % The document class has to be "report"

\usepackage{amsmath,amsthm,amsfonts}   % These are to write mathematics

\usepackage{amscd}
\usepackage{amsfonts}
\usepackage{amssymb}
\usepackage{color}      % Need the color package
\usepackage{epsfig}
\usepackage{graphicx}           % Graphics package

\usepackage{tikz}
\theoremstyle{plain}
\newtheorem{theorem}{Theorem}[section]
\newtheorem{lemma}[theorem]{Lemma}
\newtheorem{corollary}[theorem]{Corollary}
\newtheorem{proposition}[theorem]{Proposition}
\newtheorem{observation}[theorem]{Observation}
\newtheorem{remark}[theorem]{Remark}
\newtheorem{example}[theorem]{Example}

\newtheorem{question}[theorem]{Question}

\theoremstyle{definition}

\newcommand{\diam}{\textnormal{diam}}
\newcommand{\Sd}{\textnormal{Sd}}

\def\finf{\mathop{{\rm I}\kern -.27 em {\rm F}}\nolimits}

\newcommand{\rmv}[1]{}

\begin{document}

%\date{}

\title{The Simultaneous Fractional Dimension of Graph Families}

\author{{\bf{Cong X. Kang}}$^1$, {\bf{Iztok Peterin}}$^2$, and {\bf{Eunjeong Yi}}$^3$\\
\small Texas A\&M University at Galveston, Galveston, TX 77553, USA$^{1,3}$\\
\small University of Maribor, FEECS Smetanova 17, 2000 Maribor, Slovenia$^2$\\
{\small\em kangc@tamug.edu}$^1$; {\small\em iztok.peterin@um.si}$^2$; {\small\em yie@tamug.edu}$^3$}

\maketitle

\date{}

\begin{abstract}
Let $d_G(u,v)$ denote the least number of edges linking vertices $u$ and $v$ belonging to the vertex set $V$ of a graph $G$ or $\infty$ if the two vertices are not linked. Let $\diam(G)=\max\{d_G(x,y):x,y\in V\}$. For distinct vertices $x, y\in V$, let $R_G\{x,y\}=\{z\in V: d_G(x,z) \neq d_G(y,z)\}$. For a function $g$ defined on $V$ and any $U\subseteq V$, let $g(U)=\sum_{s\in U} g(s)$. A real-valued function $g:V\rightarrow[0,1]$ is a \emph{resolving function} of $G$ if $g(R_G\{x,y\}) \ge 1$ for any pair of distinct $x,y\in V$, and the \emph{fractional dimension}, $\dim_f(G)$, of $G$ is $\min\{g(V): g \mbox{ is a resolving function of } G\}$. Let $\mathcal{C}=\{G_1, G_2, \ldots, G_k\}$ be a family of connected graphs having a common vertex set $V$, where $k\ge 2$ and $|V|\ge 3$. In this paper, we introduce the simultaneous fractional dimension of a graph family. A real-valued function $h:V\rightarrow[0,1]$ is a \emph{simultaneous resolving function} of $\mathcal{C}$ if $h(R_G\{x,y\}) \ge 1$ for any distinct vertices $x,y\in V$ and for every graph $G\in \mathcal{C}$, and the \emph{simultaneous fractional dimension}, $\Sd_f(\mathcal{C})$, of $\mathcal{C}$ is $\min\{h(V): h \mbox{ is a simultaneous resolving function of }\mathcal{C}\}$. We observe that $1\le \max_{1 \le i \le k}\{\dim_f(G_i)\} \le \Sd_f(\mathcal{C}) \le \min\{\sum_{i=1}^{k}\dim_f(G_i), \frac{|V|}{2}\}$. We provide a family of graphs $\mathcal{C}$ satisfying $\Sd_f(\mathcal{C})=\max_{1 \le i \le k}\{\dim_f(G_i)\}$ and $\Sd_f(\mathcal{C})=\sum_{i=1}^{k}\dim_f(G_i)=\frac{|V|}{2}$, respectively. We characterize $\mathcal{C}$ satisfying $\Sd_f(\mathcal{C})=1$, examine $\mathcal{C}$ satisfying $\Sd_f(\mathcal{C})=\frac{|V|}{2}$, and determine $\Sd_f(\mathcal{C})$ when $\mathcal{C}$ is a family of vertex-transitive graphs. We also consider the simultaneous fractional dimension of a graph and its complement. We show that if $\{\diam(G), \diam(\overline{G})\} \neq \{3\}$ and $\diam(G) \le \diam(\overline{G})$, then $\Sd_f(G, \overline{G})=\dim_f(G)$, where $\overline{G}$ denotes the complement of $G$. We also determine $\Sd_f(G,\overline{G})$ when $G$ is a tree or a unicyclic graph. 
\end{abstract}

\noindent\small {\bf{Keywords:}} metric dimension, fractional metric dimension, resolving function, simultaneous (metric) dimension, simultaneous fractional (metric) dimension\\

\noindent\small {\bf{2010 Mathematics Subject Classification:}} 05C12

%%%%%%%%%%%%%%%%%%%%%%%%%%%%%%%%%%%%%%%%%%%%%
%%%%%%%%%%%%%%%%%%%%%%%%%%%%%%%%%%%%%%%%%%%%%

\section{Introduction}

Let $G$ be a finite, simple, and undirected graph with vertex set $V(G)$ and edge set $E(G)$. For any two vertices $x,y\in V(G)$, let $d_G(x, y)$ denote the minimum number of edges on a path connecting $x$ and $y$ in $G$, should such a path exist; otherwise, we define $d_G(x,y)=\infty$. The \emph{diameter}, $\diam(G)$, of $G$ is $\max\{ d_G(x,y): x,y \in V(G)\}$. The \emph{degree}, $\deg_G(v)$, of a vertex $v\in V(G)$ is the number of edges incident to $v$ in $G$; an \emph{end-vertex} is a vertex of degree one, a \emph{support vertex} is a vertex that is adjacent to an end-vertex, and a \emph{major vertex} is a vertex of degree at least three. The \emph{complement} of $G$, denoted by $\overline{G}$, has vertex set $V(\overline{G})=V(G)$ and edge set $E(\overline{G})$ such that $xy \in E(\overline{G})$ if and only if $xy \not\in E(G)$ for any distinct $x,y \in V(G)$. Let $P_n, C_n$ and $K_n$, respectively, denote the path, the cycle and the complete graph on $n$ vertices. For a positive integer $k$, let $[k]=\{1,2,\ldots, k\}$. 

A vertex $z \in V(G)$ \emph{resolves} a pair of distinct vertices $x$ and $y$ in $G$ if $d_G(x,z) \neq d_G(y,z)$. For two distinct vertices $x,y \in V(G)$, let $R_G\{x,y\}=\{z \in V(G): d_G(x,z) \neq d_G(y,z)\}$. A set $S \subseteq V(G)$ is a \emph{resolving set} of $G$ if $|S\cap R_G\{x,y\}| \ge 1$ for every pair of distinct vertices $x,y \in V(G)$. The \emph{metric dimension} $\dim(G)$ of $G$ is the minimum cardinality among all resolving sets of $G$. Metric dimension, introduced by  Slater~\cite{slater} and by Harary and Melter~\cite{harary}, has applications in robot navigation~\cite{tree2}, sonar~\cite{slater} and combinational optimization~\cite{sebo}, to name a few. It was noted in~\cite{NP} that determining the metric dimension of a general graph is an NP-hard problem. 

Brigham and Dutton~\cite{factor_domination} studied domination parameter for a family of edge-disjoint graphs on a common vertex set; they focused on the domination number for a graph and its complement. Recently, Ram\'{i}rez-Cruz et al.~\cite{juan} introduced the notion of a simultaneous resolving set and the simultaneous (metric) dimension as follows: for a family $\mathcal{C}=\{H_1, H_2, \ldots, H_t\}$ of (not necessarily edge-disjoint) connected graphs $H_i=(V, E_i)$ with a common vertex set $V$ (the union of whose edge sets is not necessarily the complete graph), a set $S\subseteq V$ is a \emph{simultaneous resolving set} of $\mathcal{C}$ if $S$ is a resolving set for every graph $H_i$, where $i\in[t]$, and the \emph{simultaneous dimension} of $\mathcal{C}$, denoted by $\Sd(\mathcal{C})$ or $\Sd(H_1, H_2, \ldots, H_t)$, is the minimum among the cardinalities of all simultaneous resolving sets of $\mathcal{C}$. It was shown in~\cite{juan} that, for any family $\mathcal{G}=\{G_1,\ldots, G_k\}$ of connected graphs with a common vertex set $V$, $1\le\max_{1 \le i \le k}\{\dim(G_i)\} \le \Sd(\mathcal{G}) \le \min\left\{\sum_{i=1}^{k} \dim(G_i), |V|-1\right\}\le |V|-1$.

The fractionalization of various graph parameters has been studied extensively (see~\cite{fractionalization}). Currie and Oellermann~\cite{oellermann} defined fractional (metric) dimension as the optimal solution to a linear programming problem by relaxing a condition of the integer programming problem for metric dimension. Arumugam and Mathew~\cite{fdim} officially studied the fractional dimension of graphs. For a function $g$ defined on $V(G)$ and for $U\subseteq V(G)$, let $g(U)=\sum_{s\in U} g(s)$. A real-valued function $g:V(G)\rightarrow[0,1]$ is a \emph{resolving function} of $G$ if $g(R_G\{x,y\}) \ge 1$ for any pair of distinct $x,y\in V(G)$, and the \emph{fractional dimension} $\dim_f(G)$ of $G$ is $\min\{g(V(G)): g \mbox{ is a resolving function of } G\}$. 

Analogous to the simultaneous dimension, we introduce the notion of simultaneous fractional (metric) dimension as follows. A real-valued function $g:V \rightarrow [0,1]$ is a \emph{simultaneous resolving function} for a family of graphs $\mathcal{C}=\{H_1, H_2, \ldots, H_t\}$ on a common vertex set $V$ if $g$ is a resolving function for every graph $H_i$, $i\in[t]$, and the \emph{simultaneous fractional dimension} of $\mathcal{C}$, denoted by $\Sd_f(\mathcal{C})$ or $\Sd_f(H_1, H_2, \ldots, H_t)$, is $\min\{g(V): g \mbox{ is a simultaneous resolving function of }\mathcal{C}\}$. Note that $\Sd_f(\mathcal{C})$ reduces to $\Sd(\mathcal{C})$ if the codomain of simultaneous resolving functions is restricted to $\{0,1\}$.

Next, we recall some known results on the fractional dimension of graphs. We recall some terminology. Fix a tree $T$. An end-vertex $\ell$ is called a \emph{terminal vertex} of a major vertex $v$ in $T$ if $d_T(\ell,v) < d_T(\ell,w)$ for every other major vertex $w$ in $T$. The terminal degree, $ter_T(v)$, of a major vertex $v$ is the number of terminal vertices of $v$ in $T$, and an \emph{exterior major vertex} is a major vertex that has a positive terminal degree. Let $\sigma(T)$ denote the number of end-vertices of $T$, and let $ex(T)$ denote the number of exterior major vertices of $T$.

\begin{theorem}\label{dim_frac_graph} 
\hfill{}
\begin{itemize}
\item[(a)] \emph{\cite{frac_sdim}} For any connected graph $G$ of order $n \ge 2$, $\dim_f(G)=1$ if and only if $G=P_n$.
\item[(b)] \emph{\cite{yi}} For any non-trivial tree $T$, $\dim_f (T)=\frac{1}{2}(\sigma(T)-ex_1(T))$, where $ex_1(T)$ denotes the number of exterior major vertices $u$ with $ter_T(u)=1$.
\item[(c)] \emph{\cite{fdim}} For $n\ge 3$, $\dim_f(C_n)=\left\{
\begin{array}{ll}
\frac{n}{n-1} & \mbox{ if $n$ is odd},\\
\frac{n}{n-2} & \mbox{ if $n$ is even}. 
\end{array}\right.$
\item[(d)] \emph{\cite{fdim}} For the Petersen graph $\mathcal{P}$, $\dim_f(\mathcal{P})=\frac{5}{3}$.
\item[(e)] \emph{\cite{fdim}} For the wheel graph $W_n=K_1+C_{n-1}$ (the join of $K_1$ and $C_{n-1}$) of order $n\ge 4$, 
\begin{equation*}
\dim_f(W_n)=\left\{
\begin{array}{ll}
2 & \mbox{ if } n\in\{4,5\},\\
\frac{3}{2} & \mbox{ if } n=6,\\
\frac{n-1}{4} & \mbox{ if } n \ge 7. 
\end{array}\right.
\end{equation*}
\item[(f)] \emph{\cite{frac_sdim}} If $B_m$ is a bouquet of $m\ge2$ cycles with a cut-vertex (i.e., the vertex sum of $m$ cycles at one common vertex), then $\dim_f(B_m)=m$.
\item[(g)] \emph{\cite{fdim, frac_sdim_CxK}} A connected graph $G$ of order $n\geq 2$ has $\dim_f(G)\leq \frac{n}{2}$; further, $\dim_f(G)=\frac{n}{2}$ if and only if there is a fixed-point-free permutation $\phi$ on $V(G)$ such that $|R\{v, \phi(v)\}|=2$ for each $v\in V(G)$. 
\end{itemize}
\end{theorem}

In this paper, we obtain some general results on the simultaneous fractional dimension for a family of connected graphs on a common vertex set. We also obtain some results on the simultaneous fractional dimension for a graph and its complement. The paper is organized as follow. In Section~\ref{sec_sdim_frac}, we obtain some general results on the simultaneous fractional dimension of graph families. Let $\mathcal{G}=\{G_1, G_2, \ldots, G_k\}$ be a family of connected graphs on a common vertex set $V$, where $k\ge 2$ and $|V|\ge3$. We observe that $1 \le \max_{1 \le i \le k}\{\dim_f(G_i)\} \le \Sd_f(\mathcal{G}) \le \min\left\{\sum_{i=1}^{k} \dim_f(G_i), \frac{|V|}{2}\right\} \le \frac{|V|}{2}$; we provide a family of graphs $\mathcal{G}$ satisfying $\Sd_f(\mathcal{G})=\max_{1 \le i \le k}\{\dim_f(G_i)\}$ and $\Sd_f(\mathcal{G})=\sum_{i=1}^{k}\dim_f(G_i)=\frac{|V|}{2}$, respectively. We also show that $\Sd_f(\mathcal{G})-\max_{1 \le i \le k}\{\dim_f(G_i)\}$ and $\min\left\{\sum_{i=1}^{k} \dim_f(G_i), \frac{|V|}{2}\right\}-\Sd_f(\mathcal{G})$, respectively, can be arbitrarily large. We characterize $\mathcal{G}$ satisfying $\Sd_f(\mathcal{G})=1$, and we examine $\mathcal{G}$ satisfying $\Sd_f(\mathcal{G})=\frac{|V|}{2}$. We also determine $\Sd_f(\mathcal{G})$ when $\mathcal{G}$ is a family of vertex-transitive graphs. In Section~\ref{sec_sdim_frac_complement}, we show that if $\{\diam(G), \diam(\overline{G})\} \neq \{3\}$ and $\diam(G)\le \diam(\overline{G})$, then $\Sd_f(G, \overline{G})=\dim_f(G)$. We also determine $\Sd_f(G, \overline{G})$ when $G$ is a tree or a unicyclic graph.

\section{General results on simultaneous fractional dimension}\label{sec_sdim_frac}

Throughout this section, let $\mathcal{G}=\{G_1, G_2, \ldots, G_k\}$ be a family of connected graphs with a common vertex set $V$, where $k \ge2$ and $|V| \ge 3$. In this section, we observe that $1\le\max_{1 \le i \le k}\{\dim_f(G_i)\} \le \Sd_f(\mathcal{G}) \le \min\left\{\sum_{i=1}^{k} \dim_f(G_i), \frac{|V|}{2}\right\} \le \frac{|V|}{2}$; we show the existence of $\mathcal{G}$ satisfying $\Sd_f(\mathcal{G})=\max_{1 \le i \le k}\{\dim_f(G_i)\}$ and $\Sd_f(\mathcal{G})=\min\left\{\sum_{i=1}^{k} \dim_f(G_i), \frac{|V|}{2}\right\}$, respectively. We characterize a family of graphs $\mathcal{G}$ satisfying $\Sd_f(\mathcal{G})=1$, and we examine some families of graphs $\mathcal{G}$ with $\Sd_f(\mathcal{G})=\frac{|V|}{2}$. We also show the following: (i) if $\mathcal{G}$ is a family of paths, then $\Sd_f(\mathcal{G})\in \{1, \frac{|V|}{|V|-1}\}$; (ii) if $\mathcal{G}$ is a family of vertex-transitive graphs, then $\Sd_f(\mathcal{G})=\max_{1\le i \le k}\{\dim_f(G_i)\}$; (iii) we apply (ii) to determine $\Sd_f(\mathcal{G})$ when $\mathcal{G}$ is a family of cycles or a family of the Petersen graphs, respectively.

We begin with the following result that is useful in establishing the lower bound of $\Sd_f(\mathcal{G})$. The \emph{open neighborhood} of a vertex $v \in V(G)$ is $N_G(v)=\{u \in V(G) : uv \in E(G)\}$. Two vertices $u, w \in V(G)$ are called \emph{twins} if $N_G(u)-\{w\}=N_G(w)-\{u\}$; notice that a vertex is its own twin. Hernando et al.~\cite{Hernando} observed that the twin relation is an equivalence relation and that an equivalence class under it, called a \emph{twin equivalence class}, induces either a clique or an independent set.

\begin{lemma}\label{obs_twin_frac}
Let $S$ be a twin equivalence class of a graph $G$ with $|S|=\alpha\ge 2$. Let $g:V(G) \rightarrow [0,1]$ be a resolving function for either $G$ or $\overline{G}$. Then $g(S) \ge \frac{\alpha}{2}$.
\end{lemma}

\begin{proof}
Let $S=\{s_1, s_2, \ldots, s_{\alpha}\}$ be a twin equivalence class of $G$, where $\alpha\ge 2$. Let $g:V(G) \rightarrow [0,1]$ be a resolving function for either $G$ or $\overline{G}$. Note that, for any distinct $i, j \in [\alpha]$, $R_G\{s_i, s_j\}=\{s_i, s_j\}=R_{\overline{G}}\{s_i, s_j\}$; thus, $g(s_i)+g(s_j) \ge 1$. By summing over the $\alpha \choose 2$ inequalities, we have $(\alpha-1)g(S) \ge {\alpha \choose 2}$. Therefore, $g(S) \ge \frac{\alpha}{2}$.~\hfill
\end{proof}

We note that the argument in the preceding proof generalizes one direction of part (g) of Theorem~\ref{dim_frac_graph}. Noticing that the constant function $f(v)=\frac{1}{2}$ for each $v\in V(G)$ resolves every graph and that $\dim_f(K_{|V|})=\frac{|V|}{2}$ follows from part (g) of Theorem~\ref{dim_frac_graph} by taking as $\phi$ any fixed-point-free permutation on $V$, we obtain the following bounds for  $\Sd_f(\mathcal{G})$.

\begin{observation}\label{factor_bounds2}
Any family $\mathcal{G}=\{G_1, G_2, \ldots, G_k\}$ of connected graphs on a common vertex set $V$ satisfies 
$$\max_{1 \le i \le k}\{\dim_f(G_i)\} \le \Sd_f(\mathcal{G}) \le \min\left\{\sum_{i=1}^{k} \dim_f(G_i), \frac{|V|}{2}\right\};$$
moreover, the upper bound is achieved if the graph family $\mathcal{G}$ on vertex set $V$ contains the complete graph $K_{|V|}$ as a member.
\end{observation}

\begin{remark}
The following are further examples showing the sharpness of the upper and lower bounds of Observation~\ref{factor_bounds2}. 

For the sharpness of the lower bound, let $\mathcal{G}=\{G_1, G_2, G_3\}$ be the family of graphs given in Figure~\ref{fig_lower_bound}(a) with the common vertex set $V=\cup_{i=1}^{6}\{u_i\}$. Then $\dim_f(G_1)=\frac{3}{2}$ by Theorem~\ref{dim_frac_graph}(c), $\dim_f(G_2)=\frac{3}{2}$ by Theorem~\ref{dim_frac_graph}(e), and $\dim_f(G_3)=1$ by Theorem~\ref{dim_frac_graph}(a); thus, $\Sd_f(\mathcal{G}) \ge \frac{3}{2}$ by Observation~\ref{factor_bounds2}. Let $g: V \rightarrow [0,1]$ be a function defined by $g(u_i)=\frac{1}{4}$ for each $i\in[6]$. Note that, for each $t\in[3]$ and for any distinct $i,j\in [6]$, $|R_{G_t}\{u_i, u_j\}|\ge 4$ and $g(R_{G_t}\{u_i, u_j\}) \ge 4(\frac{1}{4})=1$; thus, $g$ is a simultaneous resolving function of $\mathcal{G}$. So, $\Sd_f(\mathcal{G}) \le g(V)=\frac{3}{2}$. Therefore, $\Sd_f(\mathcal{G})=\frac{3}{2}=\max\{\dim_f(G_1), \dim_f(G_2), \dim_f(G_3)\}$.

For the sharpness of the upper bound, let $\mathcal{H}=\{H_1, H_2, H_3\}$ be the family of graphs given in Figure~\ref{fig_lower_bound}(b) with the common vertex set $V=\cup_{i=1}^{6}\{u_i\}$. Note that $\dim_f(H_1)=\frac{3}{2}$ and $\dim_f(H_3)=2$ by Theorem~\ref{dim_frac_graph}(b), and $\dim_f(H_2)=2$ by Theorem~\ref{dim_frac_graph}(f). Let $h:V \rightarrow [0,1]$ be any simultaneous resolving function of $\mathcal{H}$. Than we have the following: (i) $h(u_2)+h(u_6)\ge 1$ since $u_2$ and $u_6$ are twins in $H_1$; (ii) $h(u_4)+h(u_5)\ge 1$ since $u_4$ and $u_5$ are twins in $H_2$; (iii) $h(u_1)+h(u_3)\ge 1$ since $u_1$ and $u_3$ are twins in $H_3$. By summing over the three inequalities, we have $h(V)\ge 3=\frac{|V|}{2}$; thus, $\Sd_f(\mathcal{H}) \ge 3$. By Observation~\ref{factor_bounds2}, we have $ \Sd_f(\mathcal{H}) =3= \min\{\sum_{i=1}^{3} \dim_f(H_i), \frac{|V|}{2}\}$.
 \end{remark}

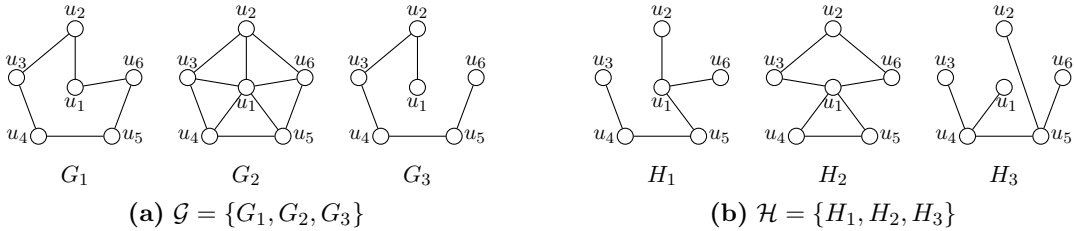
\begin{figure}[ht]
\centering
\begin{tikzpicture}[scale=.65, transform shape]

\node [draw, shape=circle, scale=0.9] (a0) at  (0.5,0) {};
\node [draw, shape=circle, scale=0.9] (a1) at  (0.5,1.2) {};
\node [draw, shape=circle, scale=0.9] (a2) at  (-0.7,0.2) {};
\node [draw, shape=circle, scale=0.9] (a3) at  (-0.25,-1) {};
\node [draw, shape=circle, scale=0.9] (a4) at  (1.25,-1) {};
\node [draw, shape=circle, scale=0.9] (a5) at  (1.7,0.2) {};

\draw(a0)--(a1)--(a2)--(a3)--(a4)--(a5)--(a0);

\node [scale=1.15] at (0.5,-0.35) {$u_1$};
\node [scale=1.15] at (0.5,1.55) {$u_2$};
\node [scale=1.15] at (-0.7,0.53) {$u_3$};
\node [scale=1.15] at (-0.68,-1) {$u_4$};
\node [scale=1.15] at (1.68,-1) {$u_5$};
\node [scale=1.15] at (1.7,0.53) {$u_6$};

\node [draw, shape=circle, scale=0.9] (b0) at  (4,0) {};
\node [draw, shape=circle, scale=0.9] (b1) at  (4,1.2) {};
\node [draw, shape=circle, scale=0.9] (b2) at  (2.8,0.2) {};
\node [draw, shape=circle, scale=0.9] (b3) at  (3.25,-1) {};
\node [draw, shape=circle, scale=0.9] (b4) at  (4.75,-1) {};
\node [draw, shape=circle, scale=0.9] (b5) at  (5.2,0.2) {};

\draw(b0)--(b1)--(b2)--(b3)--(b4)--(b5)--(b1);\draw(b2)--(b0)--(b3);\draw(b4)--(b0)--(b5);

\node [scale=1.15] at (4,-0.36) {$u_1$};
\node [scale=1.15] at (4,1.55) {$u_2$};
\node [scale=1.15] at (2.8,0.53) {$u_3$};
\node [scale=1.15] at (2.82,-1) {$u_4$};
\node [scale=1.15] at (5.18,-1) {$u_5$};
\node [scale=1.15] at (5.2,0.53) {$u_6$};

\node [draw, shape=circle, scale=0.9] (c0) at  (7.5,0) {};
\node [draw, shape=circle, scale=0.9] (c1) at  (7.5,1.2) {};
\node [draw, shape=circle, scale=0.9] (c2) at  (6.3,0.2) {};
\node [draw, shape=circle, scale=0.9] (c3) at  (6.75,-1) {};
\node [draw, shape=circle, scale=0.9] (c4) at  (8.25,-1) {};
\node [draw, shape=circle, scale=0.9] (c5) at  (8.7,0.2) {};

\draw(c0)--(c1)--(c2)--(c3)--(c4)--(c5);

\node [scale=1.15] at (7.5,-0.35) {$u_1$};
\node [scale=1.15] at (7.5,1.55) {$u_2$};
\node [scale=1.15] at (6.3,0.53) {$u_3$};
\node [scale=1.15] at (6.32,-1) {$u_4$};
\node [scale=1.15] at (8.68,-1) {$u_5$};
\node [scale=1.15] at (8.7,0.53) {$u_6$};

\node [scale=1.25] at (0.5,-1.8) {$G_1$};
\node [scale=1.25] at (4,-1.8) {$G_2$};
\node [scale=1.25] at (7.5,-1.8) {$G_3$};

\node [draw, shape=circle, scale=0.9] (d0) at  (12.5,0) {};
\node [draw, shape=circle, scale=0.9] (d1) at  (12.5,1.2) {};
\node [draw, shape=circle, scale=0.9] (d2) at  (11.3,0.2) {};
\node [draw, shape=circle, scale=0.9] (d3) at  (11.75,-1) {};
\node [draw, shape=circle, scale=0.9] (d4) at  (13.25,-1) {};
\node [draw, shape=circle, scale=0.9] (d5) at  (13.7,0.2) {};

\draw(d0)--(d4)--(d3)--(d2);\draw(d1)--(d0)--(d5);

\node [scale=1.15] at (12.5,-0.35) {$u_1$};
\node [scale=1.15] at (12.5,1.55) {$u_2$};
\node [scale=1.15] at (11.3,0.53) {$u_3$};
\node [scale=1.15] at (11.32,-1) {$u_4$};
\node [scale=1.15] at (13.68,-1) {$u_5$};
\node [scale=1.15] at (13.7,0.53) {$u_6$};

\node [draw, shape=circle, scale=0.9] (e0) at  (16,0) {};
\node [draw, shape=circle, scale=0.9] (e1) at  (16,1.2) {};
\node [draw, shape=circle, scale=0.9] (e2) at  (14.8,0.2) {};
\node [draw, shape=circle, scale=0.9] (e3) at  (15.25,-1) {};
\node [draw, shape=circle, scale=0.9] (e4) at  (16.75,-1) {};
\node [draw, shape=circle, scale=0.9] (e5) at  (17.2,0.2) {};

\draw(e0)--(e2)--(e1)--(e5)--(e0);\draw(e0)--(e3)--(e4)--(e0);

\node [scale=1.15] at (16,-0.36) {$u_1$};
\node [scale=1.15] at (16,1.55) {$u_2$};
\node [scale=1.15] at (14.8,0.53) {$u_3$};
\node [scale=1.15] at (14.82,-1) {$u_4$};
\node [scale=1.15] at (17.18,-1) {$u_5$};
\node [scale=1.15] at (17.2,0.53) {$u_6$};

\node [draw, shape=circle, scale=0.9] (f0) at  (19.5,0) {};
\node [draw, shape=circle, scale=0.9] (f1) at  (19.5,1.2) {};
\node [draw, shape=circle, scale=0.9] (f2) at  (18.3,0.2) {};
\node [draw, shape=circle, scale=0.9] (f3) at  (18.75,-1) {};
\node [draw, shape=circle, scale=0.9] (f4) at  (20.25,-1) {};
\node [draw, shape=circle, scale=0.9] (f5) at  (20.7,0.2) {};

\draw(f0)--(f3)--(f2);\draw(f1)--(f4)--(f5);\draw(f3)--(f4);

\node [scale=1.15] at (19.55,-0.35) {$u_1$};
\node [scale=1.15] at (19.5,1.55) {$u_2$};
\node [scale=1.15] at (18.3,0.53) {$u_3$};
\node [scale=1.15] at (18.32,-1) {$u_4$};
\node [scale=1.15] at (20.68,-1) {$u_5$};
\node [scale=1.15] at (20.7,0.53) {$u_6$};

\node [scale=1.25] at (12.5,-1.8) {$H_1$};
\node [scale=1.25] at (16,-1.8) {$H_2$};
\node [scale=1.25] at (19.5,-1.8) {$H_3$};

\node [scale=1.35] at (4,-2.6) {\textbf{(a)} $\mathcal{G}=\{G_1, G_2, G_3\}$};
\node [scale=1.35] at (16,-2.6) {\textbf{(b)} $\mathcal{H}=\{H_1, H_2, H_3\}$};

\end{tikzpicture}
\caption{\small{(a) A family $\mathcal{G}=\{G_1, G_2, G_3\}$ with $\Sd_f(\mathcal{G})=\max_{1 \le i \le 3}\{\dim_f(G_i)\}$; (b) A family $\mathcal{H}=\{H_1, H_2, H_3\}$ with a common vertex set $V$ such that $\Sd_f(\mathcal{H})=\min\{\sum_{i=1}^3\dim_f(H_i), \frac{|V|}{2}\}$.}}\label{fig_lower_bound}
\end{figure}

Next, we characterize a family of connected graphs $\mathcal{G}$ satisfying $\Sd_f(\mathcal{G})=1$. We recall the following result.

\begin{theorem}\emph{\cite{juan}}\label{thm_bounds}
Let $\mathcal{G}$ be a family of connected graphs on a common vertex set $V$. Then
\begin{itemize}
\item[(a)] $\Sd(\mathcal{G})=1$ if and only if $\mathcal{G}$ is a family of paths that share a common end-vertex.
\item[(b)] $\Sd(\mathcal{G})=|V|-1$ if and only if, for every pair of distinct $x,y\in V$, there exists a graph $G\in \mathcal{G}$ such that $x$ and $y$ are twins in $G$.
\end{itemize}
\end{theorem}

\begin{theorem}\label{sdim_frac_path}
Let $\mathcal{G}$ be a family of connected graphs on a common vertex set. Then $\Sd_f(\mathcal{G})=1$ if and only if $\mathcal{G}$ is a collection of paths that share a common end-vertex.
\end{theorem}

\begin{proof}
Let $\mathcal{G}=\{G_1, G_2, \ldots, G_k\}$ be a family of connected graphs on a common vertex set $V$.

($\Leftarrow$) Let $\mathcal{G}$ be a family of paths that share a common end-vertex. Since $1\le \Sd_f(\mathcal{G}) \le \Sd(\mathcal{G})$, Theorem~\ref{thm_bounds}(a) implies that $\Sd_f(\mathcal{G})=1$.

($\Rightarrow$) Let $\Sd_f(\mathcal{G})=1$; then $\dim_f(G_i)=1$ for each $i\in[k]$. By Theorem~\ref{dim_frac_graph}(a), $G_i$ must be a path for each $i\in[k]$. Suppose that there is no common end-vertex for all paths in $\mathcal{G}$ with $V=\{u_1, u_2, \ldots, u_n\}$, where $n\ge 3$. Then, for each $u_i\in V$, we can choose a graph $G_j \in \mathcal{G}$ such that $\deg_{G_j}(u_i)=2$, where $i\in[n]$ and $j=j(i)\in[k]$. Let $h: V \rightarrow [0,1]$ be any simultaneous resolving function of $\mathcal{G}$. Since $R_{G_j}(N(u_i))=V-\{u_i\}$, we have $h(V)-h(u_i) \ge 1$ for each $i\in[n]$. By summing over the $n$ inequalities, we have $(n-1)h(V)\ge n$, i.e., $h(V) \ge \frac{n}{n-1}$; thus, $\Sd_f(\mathcal{G})\ge \frac{|V|}{|V|-1}>1$, contradicting the assumption that $\Sd_f(\mathcal{G})=1$. So, $\mathcal{G}$ must be a collection of paths which share a common end-vertex.~\hfill
\end{proof}

Based on Theorem~\ref{sdim_frac_path} and its proof, together with the fact that $h: V\rightarrow [0,1]$ defined by $h(u)=\frac{1}{|V|-1}$ for each $u\in V$ forms a simultaneous resolving function for a family of paths, we have the following.

\begin{corollary}\label{cor_path_new}
Let $\mathcal{G}$ be a family of paths on a common vertex set $V$. Then 
\begin{equation*}
\Sd_f(\mathcal{G})=\left\{
\begin{array}{ll}
1 & \mbox{if all members of $\mathcal{G}$ share a common end-vertex},\\
\frac{|V|}{|V|-1} & \mbox{otherwise}.
\end{array}
\right.
\end{equation*}
\end{corollary}

Next, we provide three families $\mathcal{G}=\{G_1, \ldots, G_k\}$ of connected graphs on a common vertex set $V$ with $K_{|V|}\notin \mathcal{G}$.

\begin{example}\label{sdim_frac_n2}
\begin{itemize}
\item[(a)] There exists $\mathcal{G}\!=\!\{G_1, \ldots, G_k\}$ satisfying $\Sd_f(\mathcal{G})=\sum_{i=1}^{k}\dim_f(G_i)=\frac{|V|}{2}$.
\item[(b)] There exists $\mathcal{G}=\{G_1, \ldots, G_k\}$ with $\Sd_f(\mathcal{G})=\frac{|V|}{2}$ such that each vertex in $V$ belongs to a twin equivalence class of cardinality at least two in some $G_i$, where $i\in[k]$.
\item[(c)] There exists $\mathcal{G}=\{G_1, \ldots, G_k\}$ with $\Sd_f(\mathcal{G})<\frac{|V|}{2}$ such that each vertex in $V$ belongs to a twin equivalence class of cardinality at least two in some $G_i$, where $i\in[k]$.
\end{itemize}
\end{example}

\begin{proof}
Let $g:V \rightarrow [0,1]$ be any simultaneous resolving function of $\mathcal{G}$.

(a) Let $\mathcal{G}=\{G_1, G_2, G_3\}$ be the family of trees in Figure~\ref{fig_upperbound1}. By Theorem~\ref{dim_frac_graph}(b),  $\dim_f(G_i)=2$ for each $i\in[3]$. By Lemma~\ref{obs_twin_frac}, we have the following: (i) $g(u_1)+g(u_2)\ge 1$ and $g(u_7)+g(u_8)\ge 1$ since $\{u_1, u_2\}$ and $\{u_7, u_8\}$ are twin equivalence classes in $G_1$; (ii) $g(u_3)+g(u_4)\ge 1$ and $g(u_{11})+g(u_{12})\ge 1$ since $\{u_3, u_4\}$ and $\{u_{11}, u_{12}\}$ are twin equivalence classes in $G_2$; (iii) $g(u_5)+g(u_6)\ge 1$ and $g(u_9)+g(u_{10})\ge 1$ since $\{u_5, u_6\}$ and $\{u_9, u_{10}\}$ are twin equivalence classes in $G_3$. By summing over the six inequalities, we have $g(V)\ge 6$; thus, $\Sd_f(\mathcal{G})\ge 6$. Since $\Sd_f(\mathcal{G})\le 6=\frac{|V|}{2}$ by Observation~\ref{factor_bounds2}, $\Sd_f(\mathcal{G})=6=\frac{|V|}{2}=\sum_{i=1}^{3}\dim_f(G_i)$.

(b) Let $\mathcal{G}=\{G_1, G_2, G_3, G_4, G_5\}$ be the family of trees in Figure~\ref{fig_upperbound2}. By Lemma~\ref{obs_twin_frac}, we have the following: (i) $g(u_1)+g(u_2)\ge 1$ since $\{u_1,u_2\}$ is a twin equivalence class in $G_1$; (ii) $g(u_2)+g(u_3)\ge 1$ since $\{u_2,u_3\}$ is a twin equivalence class in $G_2$; (iii) $g(u_3)+g(u_4)\ge 1$ since $\{u_3,u_4\}$ is a twin equivalence class in $G_3$; (iv) $g(u_4)+g(u_5)\ge 1$ since $\{u_4,u_5\}$ is a twin equivalence class in $G_4$; (v) $g(u_5)+g(u_1)\ge 1$ since $\{u_5,u_1\}$ is a twin equivalence class in $G_5$. By summing over the five inequalities, we have $2g(V)\ge 5$, i.e., $g(V) \ge \frac{5}{2}$; thus, $\Sd_f(\mathcal{G}) \ge \frac{5}{2}=\frac{|V|}{2}$. Since $\Sd_f(\mathcal{G})\le \frac{5}{2}=\frac{|V|}{2}$ by Observation~\ref{factor_bounds2}, $\Sd_f(\mathcal{G})=\frac{|V|}{2}$.

(c) Let $\mathcal{G}=\{G_1, G_2, G_4\}$ be the family of trees in Figure~\ref{fig_upperbound2}. Then $\{u_1,u_2\}$ is a twin equivalence class in $G_1$, $\{u_2,u_3\}$ is a twin equivalence class in $G_2$, and $\{u_4,u_5\}$ is a twin equivalence class in $G_4$; thus, each vertex in $V$ belongs to a twin equivalence class of cardinality at least two in some $G_i$, where $i\in\{1,2,4\}$. We note that, by Lemma~\ref{obs_twin_frac},  $g(u_1)+g(u_2)\ge 1$ and $g(u_4)+g(u_5)\ge 1$, and thus $g(V) \ge 2$. So, $\Sd_f(\mathcal{G})\ge 2$. Now, let $h:V \rightarrow[0,1]$ be a function defined by $h(u_2)=1$, $h(u_4)=\frac{1}{2}=h(u_5)$, and $h(u_1)=0=h(u_3)$. Then $h$ forms a simultaneous resolving function for $\mathcal{G}$ with $h(V)=2$; thus, $\Sd_f(\mathcal{G})\le 2$. Therefore, $\Sd_f(\mathcal{G})=2<\frac{|V|}{2}$.~\hfill
\end{proof}

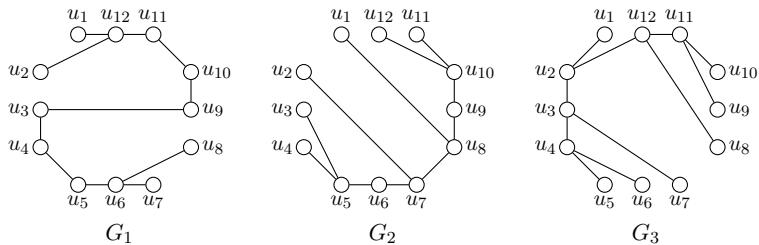
\begin{figure}[ht]
\centering
\begin{tikzpicture}[scale=.5, transform shape]

\node [draw, shape=circle, scale=1.1] (a2) at  (-1,3) {};
\node [draw, shape=circle, scale=1.1] (a3) at  (-1,2) {};
\node [draw, shape=circle, scale=1.1] (a4) at  (-1,1) {};
\node [draw, shape=circle, scale=1.1] (a5) at  (0,0) {};
\node [draw, shape=circle, scale=1.1] (a6) at  (1,0) {};
\node [draw, shape=circle, scale=1.1] (a7) at  (2,0) {};
\node [draw, shape=circle, scale=1.1] (a10) at  (3,3) {};
\node [draw, shape=circle, scale=1.1] (a9) at  (3,2) {};
\node [draw, shape=circle, scale=1.1] (a8) at  (3,1) {};
\node [draw, shape=circle, scale=1.1] (a1) at  (0,4) {};
\node [draw, shape=circle, scale=1.1] (a12) at  (1,4) {};
\node [draw, shape=circle, scale=1.1] (a11) at  (2,4) {};

\draw(a1)--(a12)--(a2);\draw(a12)--(a11)--(a10)--(a9)--(a3)--(a4)--(a5)--(a6);\draw(a7)--(a6)--(a8);

\node [scale=1.5] at (0,4.49) {$u_1$};
\node [scale=1.5] at (-1.57,3) {$u_2$};
\node [scale=1.5] at (-1.57,2) {$u_3$};
\node [scale=1.5] at (-1.57,1) {$u_4$};
\node [scale=1.5] at (0,-0.45) {$u_5$};
\node [scale=1.5] at (1,-0.45) {$u_6$};
\node [scale=1.5] at (2,-0.45) {$u_7$};
\node [scale=1.5] at (3.59,1) {$u_8$};
\node [scale=1.5] at (3.59,2) {$u_9$};
\node [scale=1.5] at (3.69,3) {$u_{10}$};
\node [scale=1.5] at (2,4.49) {$u_{11}$};
\node [scale=1.5] at (1,4.49) {$u_{12}$};

\node [draw, shape=circle, scale=1.1] (b2) at  (6,3) {};
\node [draw, shape=circle, scale=1.1] (b3) at  (6,2) {};
\node [draw, shape=circle, scale=1.1] (b4) at  (6,1) {};
\node [draw, shape=circle, scale=1.1] (b5) at  (7,0) {};
\node [draw, shape=circle, scale=1.1] (b6) at  (8,0) {};
\node [draw, shape=circle, scale=1.1] (b7) at  (9,0) {};
\node [draw, shape=circle, scale=1.1] (b10) at  (10,3) {};
\node [draw, shape=circle, scale=1.1] (b9) at  (10,2) {};
\node [draw, shape=circle, scale=1.1] (b8) at  (10,1) {};
\node [draw, shape=circle, scale=1.1] (b1) at  (7,4) {};
\node [draw, shape=circle, scale=1.1] (b12) at  (8,4) {};
\node [draw, shape=circle, scale=1.1] (b11) at  (9,4) {};

\draw(b3)--(b5)--(b4);\draw(b5)--(b6)--(b7)--(b8)--(b9)--(b10);\draw(b2)--(b7);\draw(b1)--(b8);\draw(b11)--(b10)--(b12);

\node [scale=1.5] at (7,4.49) {$u_1$};
\node [scale=1.5] at (5.43,3) {$u_2$};
\node [scale=1.5] at (5.43,2) {$u_3$};
\node [scale=1.5] at (5.43,1) {$u_4$};
\node [scale=1.5] at (7,-0.45) {$u_5$};
\node [scale=1.5] at (8,-0.45) {$u_6$};
\node [scale=1.5] at (9,-0.45) {$u_7$};
\node [scale=1.5] at (10.59,1) {$u_8$};
\node [scale=1.5] at (10.59,2) {$u_9$};
\node [scale=1.5] at (10.69,3) {$u_{10}$};
\node [scale=1.5] at (9,4.49) {$u_{11}$};
\node [scale=1.5] at (8,4.49) {$u_{12}$};

\node [draw, shape=circle, scale=1.1] (c2) at  (13,3) {};
\node [draw, shape=circle, scale=1.1] (c3) at  (13,2) {};
\node [draw, shape=circle, scale=1.1] (c4) at  (13,1) {};
\node [draw, shape=circle, scale=1.1] (c5) at  (14,0) {};
\node [draw, shape=circle, scale=1.1] (c6) at  (15,0) {};
\node [draw, shape=circle, scale=1.1] (c7) at  (16,0) {};
\node [draw, shape=circle, scale=1.1] (c10) at  (17,3) {};
\node [draw, shape=circle, scale=1.1] (c9) at  (17,2) {};
\node [draw, shape=circle, scale=1.1] (c8) at  (17,1) {};
\node [draw, shape=circle, scale=1.1] (c1) at  (14,4) {};
\node [draw, shape=circle, scale=1.1] (c12) at  (15,4) {};
\node [draw, shape=circle, scale=1.1] (c11) at  (16,4) {};

\draw(c5)--(c4)--(c6);\draw(c4)--(c3)--(c2)--(c12)--(c11);\draw(c3)--(c7);\draw(c1)--(c2);\draw(c8)--(c12);\draw(c9)--(c11)--(c10);

\node [scale=1.5] at (14,4.49) {$u_1$};
\node [scale=1.5] at (12.43,3) {$u_2$};
\node [scale=1.5] at (12.43,2) {$u_3$};
\node [scale=1.5] at (12.43,1) {$u_4$};
\node [scale=1.5] at (14,-0.45) {$u_5$};
\node [scale=1.5] at (15,-0.45) {$u_6$};
\node [scale=1.5] at (16,-0.45) {$u_7$};
\node [scale=1.5] at (17.59,1) {$u_8$};
\node [scale=1.5] at (17.59,2) {$u_9$};
\node [scale=1.5] at (17.69,3) {$u_{10}$};
\node [scale=1.5] at (16,4.49) {$u_{11}$};
\node [scale=1.5] at (15,4.49) {$u_{12}$};

\node [scale=1.6] at (1.1,-1.3) {$G_1$};
\node [scale=1.6] at (8.1,-1.3) {$G_2$};
\node [scale=1.6] at (15.1,-1.3) {$G_3$};
\end{tikzpicture}
\caption{\small{A family $\mathcal{G}=\{G_1, G_2, G_3\}$ on a common vertex set $V$ with $\Sd_f(\mathcal{G})=\sum_{i=1}^{3}\dim_f(G_i)=\frac{|V|}{2}$.}}\label{fig_upperbound1}
\end{figure}

\begin{figure}[ht]
\centering
\begin{tikzpicture}[scale=.65, transform shape]

\node [draw, shape=circle, scale=0.9] (a1) at  (0.5,1.1) {};
\node [draw, shape=circle, scale=0.9] (a2) at  (-0.55,0.3) {};
\node [draw, shape=circle, scale=0.9] (a3) at  (-0.25,-0.8) {};
\node [draw, shape=circle, scale=0.9] (a4) at  (1.25,-0.8) {};
\node [draw, shape=circle, scale=0.9] (a5) at  (1.55,0.3) {};

\draw(a1)--(a3)--(a2);\draw(a3)--(a4)--(a5);

\node [scale=1.15] at (0.5,1.45) {$u_1$};
\node [scale=1.15] at (-0.98,0.3) {$u_2$};
\node [scale=1.15] at (-0.25,-1.2) {$u_3$};
\node [scale=1.15] at (1.25,-1.2) {$u_4$};
\node [scale=1.15] at (1.98,0.3) {$u_5$};

\node [draw, shape=circle, scale=0.9] (b1) at  (4.5,1.1) {};
\node [draw, shape=circle, scale=0.9] (b2) at  (3.45,0.3) {};
\node [draw, shape=circle, scale=0.9] (b3) at  (3.75,-0.8) {};
\node [draw, shape=circle, scale=0.9] (b4) at  (5.25,-0.8) {};
\node [draw, shape=circle, scale=0.9] (b5) at  (5.55,0.3) {};

\draw(b2)--(b1)--(b3);\draw(b1)--(b4)--(b5);

\node [scale=1.15] at (4.5,1.45) {$u_1$};
\node [scale=1.15] at (3.02,0.3) {$u_2$};
\node [scale=1.15] at (3.75,-1.2) {$u_3$};
\node [scale=1.15] at (5.25,-1.2) {$u_4$};
\node [scale=1.15] at (5.98,0.3) {$u_5$};

\node [draw, shape=circle, scale=0.9] (c1) at  (8.5,1.1) {};
\node [draw, shape=circle, scale=0.9] (c2) at  (7.45,0.3) {};
\node [draw, shape=circle, scale=0.9] (c3) at  (7.75,-0.8) {};
\node [draw, shape=circle, scale=0.9] (c4) at  (9.25,-0.8) {};
\node [draw, shape=circle, scale=0.9] (c5) at  (9.55,0.3) {};

\draw(c3)--(c2)--(c4);\draw(c5)--(c1)--(c2);

\node [scale=1.15] at (8.5,1.45) {$u_1$};
\node [scale=1.15] at (7.02,0.3) {$u_2$};
\node [scale=1.15] at (7.75,-1.2) {$u_3$};
\node [scale=1.15] at (9.25,-1.2) {$u_4$};
\node [scale=1.15] at (9.98,0.3) {$u_5$};

\node [draw, shape=circle, scale=0.9] (d1) at  (12.5,1.1) {};
\node [draw, shape=circle, scale=0.9] (d2) at  (11.45,0.3) {};
\node [draw, shape=circle, scale=0.9] (d3) at  (11.75,-0.8) {};
\node [draw, shape=circle, scale=0.9] (d4) at  (13.25,-0.8) {};
\node [draw, shape=circle, scale=0.9] (d5) at  (13.55,0.3) {};

\draw(d4)--(d1)--(d5);\draw(d1)--(d2)--(d3);

\node [scale=1.15] at (12.5,1.45) {$u_1$};
\node [scale=1.15] at (11.02,0.3) {$u_2$};
\node [scale=1.15] at (11.75,-1.2) {$u_3$};
\node [scale=1.15] at (13.25,-1.2) {$u_4$};
\node [scale=1.15] at (13.98,0.3) {$u_5$};

\node [draw, shape=circle, scale=0.9] (e1) at  (16.5,1.1) {};
\node [draw, shape=circle, scale=0.9] (e2) at  (15.45,0.3) {};
\node [draw, shape=circle, scale=0.9] (e3) at  (15.75,-0.8) {};
\node [draw, shape=circle, scale=0.9] (e4) at  (17.25,-0.8) {};
\node [draw, shape=circle, scale=0.9] (e5) at  (17.55,0.3) {};

\draw(e5)--(e4)--(e1);\draw(e2)--(e3)--(e4);

\node [scale=1.15] at (16.5,1.45) {$u_1$};
\node [scale=1.15] at (15.02,0.3) {$u_2$};
\node [scale=1.15] at (15.75,-1.2) {$u_3$};
\node [scale=1.15] at (17.25,-1.2) {$u_4$};
\node [scale=1.15] at (17.98,0.3) {$u_5$};

\node [scale=1.2] at (0.5,-1.8) {$G_1$};
\node [scale=1.2] at (4.5,-1.8) {$G_2$};
\node [scale=1.2] at (8.5,-1.8) {$G_3$};
\node [scale=1.2] at (12.5,-1.8) {$G_4$};
\node [scale=1.2] at (16.5,-1.8) {$G_5$};

\end{tikzpicture}
\caption{\small{Two families $\mathcal{G}=\{G_1, G_2, G_3, G_4, G_5\}$ and $\mathcal{H}=\{G_1, G_2, G_4\}$ of trees on a common vertex set $V$ such that $\Sd_f(\mathcal{G})=\frac{|V|}{2}$ and $\Sd_f(\mathcal{H})<\frac{|V|}{2}$, where each vertex in $V$ belongs to a twin equivalence class of cardinality at least two in some member of the family of graphs.}}\label{fig_upperbound2}
\end{figure}
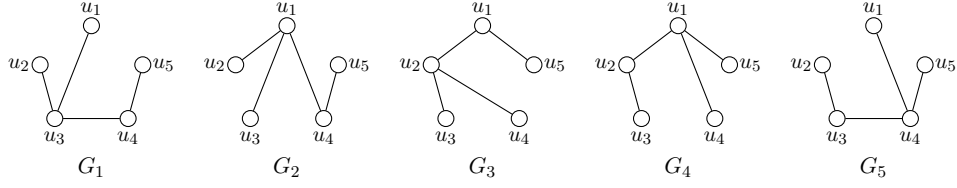

\begin{question}
Let $\mathcal{G}$ be a family of connected graphs on a common vertex set $V$. By Theorem~\ref{thm_bounds}(b) and Lemma~\ref{obs_twin_frac}, one can readily verify that $\Sd(\mathcal{G})=|V|-1$ implies $\Sd_f(\mathcal{G})=\frac{|V|}{2}$. But, as shown in Example~\ref{sdim_frac_n2}(b), there exists $\mathcal{G}$ such that $\Sd_f(\mathcal{G})=\frac{|V|}{2}$ and $\Sd(\mathcal{G})<|V|-1$. Can we characterize $\mathcal{G}$ satisfying $\Sd_f(\mathcal{G})=\frac{|V|}{2}$?
\end{question}

Let us call a twin equivalence class nontrivial if it contains two or more members. Given a graph family $\mathcal{G}=\{G_1, G_2, \ldots, G_k\}$, let $E_i$ denote the set of nontrivial twin equivalence classes of graph $G_i$. We will say a vertex $u\in V$ belongs to $E_i$ if $u$ belongs to a class in $E_i$.  For each $u\in V$, define $m_\mathcal{G}(u)$ to be the number of $E_i$'s to which $u$ belongs. 

 \begin{proposition}
Given $\mathcal{G}=\{G_1, G_2, \ldots, G_k\}$, $\Sd_f(\mathcal{G})=\frac{|V|}{2}$ if the function $m_\mathcal{G}$ is a positive constant function on $V$.  
\end{proposition}

\begin{proof}
Let $m_\mathcal{G}=m\geq 1$ and $h$ be a minimum simultaneous resolving function of $\mathcal{G}$. Then, $h(E_i)\geq \frac{|E_i|}{2}$ by Lemma~\ref{obs_twin_frac}. Summing the $k$ inequalities (one for each $E_i$), while noting that each vertex of $V$ appears $m$ times, we get $mh(V)\geq \frac{m|V|}{2}$; thus, $h(V)\ge\frac{|V|}{2}$. Since $h(V)\le \frac{|V|}{2}$ by Observation~\ref{factor_bounds2}, we have $h(V)=\frac{|V|}{2}$ and hence $\Sd_f(\mathcal{G})=\frac{|V|}{2}$.~\hfill
\end{proof}

Next, we determine $\Sd_f(\mathcal{G})$ when $\mathcal{G}$ is a family of vertex-transitive graphs. Following~\cite{Feng}, let $r(G)=\min\{|R_G\{x,y\}|:x,y\in V(G) \mbox{ and } x\neq y\}$. We recall the following result on the fractional dimension of vertex-transitive graphs. 

\begin{theorem}\emph{\cite{Feng}}\label{Feng_vtransitive}
If $G$ is a vertex-transitive graph, then $\dim_f(G)=\frac{|V(G)|}{r(G)}$. Moreover, $g:V(G) \rightarrow [0,1]$ defined by $g(u)=\frac{1}{r(G)}$, for each $u\in V(G)$, is a minimum resolving function of $G$.
\end{theorem}

\begin{proposition}\label{sdim_frac_vtransitive}
Let $\mathcal{G}=\{G_1, G_2, \ldots, G_k\}$ be a family of vertex-transitive graphs with a common vertex set $V$. Then $\Sd_f(\mathcal{G})=\max_{1\le i \le k}\{\dim_f(G_i)\}$.
\end{proposition}

\begin{proof}
By relabeling the members of $\mathcal{G}$ if necessary, let $\dim_f(G_1)\ge \dim_f(G_2) \ge \ldots \ge \dim_f(G_k)$, and let $r=r(G_1)=\min\{|R_{G_1}\{x,y\}|:x,y\in V \mbox{ and } x\neq y\}$. Then $\Sd_f(\mathcal{G}) \ge \dim_f(G_1)$ by Observation~\ref{factor_bounds2}. If $g: V \rightarrow [0,1]$ is a function defined by $g(u)=\frac{1}{r}$ for each $u\in V$, then $g$ is a resolving function for $G_1$ (and thus for $G_i$ for each $i\in[k]$) with $g(V)=\frac{|V|}{r}$. So, $\Sd_f(\mathcal{G}) \le \frac{|V|}{r}=\dim_f(G_1)$ by Theorem~\ref{Feng_vtransitive}. Therefore, $\Sd_f(\mathcal{G})=\dim_f(G_1)=\max_{1\le i \le k}\{\dim_f(G_i)\}$.~\hfill
\end{proof}

Proposition~\ref{sdim_frac_vtransitive}, together with (c) and (d) of Theorem~\ref{dim_frac_graph}, implies the following. 

\begin{corollary}
\hfill{}
\begin{itemize}
\item[(a)] If $\mathcal{G}$ is a family of odd cycles on a common vertex set $V$, then $\Sd_f(\mathcal{G})=\frac{|V|}{|V|-1}$.
\item[(b)] If $\mathcal{G}$ is a family of even cycles on a common vertex set $V$, then $\Sd_f(\mathcal{G})=\frac{|V|}{|V|-2}$. 
\item[(c)] If $\mathcal{G}$ is a family of the Petersen graphs on a common vertex set $V$ with $|V|=10$, then $\Sd_f(\mathcal{G})=\frac{5}{3}$. 
\end{itemize}
\end{corollary}

We conclude this section with some observations showing that the gap between $\Sd_f(\mathcal{G})$ and its lower or upper bound can be arbitrarily large. Also, noting that $\Sd(\mathcal{G}) \ge \Sd_f(\mathcal{G})$, we show that $\Sd(\mathcal{G})-\Sd_f(\mathcal{G})$ can be arbitrarily large.

\begin{remark}
(a) There exists a family of graphs $\mathcal{G}=\{G_1, \ldots, G_k\}$ such that $\Sd_f(\mathcal{G})-\max_{1\le i \le k}\{\dim_f(G_i)\}$ can be arbitrarily large. Let $V=\cup_{i=1}^{k}\{u_{i,0}, u_{i,1}, u_{i,2}\}$ with $|V|=3k$, where $k\ge 3$. For each $i\in[k]$, let $G_i$ be a tree with $ex(G_i)=1$ such that $u_{i,0}$ is the exterior major vertex of $G_i$ with $ter_{G_i}(u_{i,0})=3$; for some $j\in[k]$ with $j\neq i$, let $N_{G_i}(u_{i,0})=\{u_{i,1}, u_{i,2}, u_{j,0}\}$ and let $\{u_{i,1}, u_{i,2}, u_{j,2}\}$ be the set of terminal vertices of $u_{i,0}$ in $G_i$. Then $\Sd_f(\mathcal{G}) \ge k$ by Lemma~\ref{obs_twin_frac} since $u_{i,1}$ and $u_{i,2}$ are twins in $G_i$ for each $i\in[k]$, and $\dim_f(G_i)=\frac{3}{2}$, for each $i\in[k]$, by Theorem~\ref{dim_frac_graph}(b). Thus, $\Sd_f(\mathcal{G})-\max_{1\le i \le k}\{\dim_f(G_i)\}\ge k-\frac{3}{2} \rightarrow \infty$ as $k\rightarrow \infty$.

(b) There exists a family of graphs $\mathcal{G}=\{G_1, \ldots, G_k\}$ such that $\min\{\sum_{i=1}^{k}\dim_f(G_i), \frac{|V|}{2}\}-\Sd_f(\mathcal{G})$ can be arbitrarily large. Let $V=\{w_0, w_1, w_2\} \cup \{u_1, u_2, \ldots, u_{k}\}$, where $k\ge 3$. For each $i\in[k]$, let $G_i$ be a tree with $ex(G_i)=1$ such that $w_0$ is the exterior major vertex of $G_i$ with $ter_{G_i}(w_0)=3$; let $N_{G_i}(w_0)=\{w_1, w_2, u_i\}$ and let $w_1, w_2, u_j$ be the terminal vertices of $w_0$ in $G_i$, where $j\in[k]$ and $j \neq i$. First, we show that $\Sd_f(\mathcal{G})=\frac{3}{2}$. Let $h: V \rightarrow [0,1]$ be a function defined by $h(w_1)=h(w_2)=h(u_1)=\frac{1}{2}$ and $h(v)=0$ for each $v\in V-\{w_1, w_2, u_1\}$. Then $h$ is a resolving function of $G_i$ with $h(V)=\frac{3}{2}$ for each $i\in[k]$; thus, $h$ is a simultaneous resolving function for $\mathcal{G}$, and hence $\Sd_f(\mathcal{G}) \le \frac{3}{2}$. Since $\dim_f(G_i)=\frac{3}{2}$ for each $i\in[k]$ by Theorem~\ref{dim_frac_graph}(b), we have $\Sd_f(\mathcal{G})\ge \frac{3}{2}$ by Observation~\ref{factor_bounds2}. Thus, $\Sd_f(\mathcal{G})=\frac{3}{2}$. Second, we show that $\min\{\sum_{i=1}^{k}\dim_f(G_i), \frac{|V|}{2}\}=\frac{k+3}{2}$. By Theorem~\ref{dim_frac_graph}(b), $\dim_f(G_i)=\frac{3}{2}$ for each $i\in[k]$. Since $\frac{|V|}{2}=\frac{k+3}{2}< \frac{3k}{2}=\sum_{i=1}^{k}\dim_f(G_i)$ for $k \ge 3$, $\min\{\sum_{i=1}^{k}\dim_f(G_i), \frac{|V|}{2}\}=\frac{k+3}{2}$. Therefore, $\min\{\sum_{i=1}^{k}\dim_f(G_i), \frac{|V|}{2}\}-\Sd_f(\mathcal{G})=\frac{k}{2} \rightarrow \infty$ as $k \rightarrow \infty$. 

(c) There exists a family of graphs $\mathcal{G}=\{G_1, \ldots, G_k\}$ such that $\Sd(\mathcal{G})-\Sd_f(\mathcal{G})$ can be arbitrarily large. Let $V=\{u_1, u_2, \ldots, u_k\}$ with $|V|=k\ge4$. For each $i\in [k]$, let $G_i$ be the star $K_{1, k-1}$ such that $u_i$ is the central vertex of degree $k-1$ in $G_i$. For any distinct $u_i, u_j\in V$, there exists $G_x\in \mathcal{G}$ with $x\not\in\{i,j\}$ such that $u_i$ and $u_j$ are twins in $G_x$; thus, $\Sd(\mathcal{G})=k-1$ by Theorem~\ref{thm_bounds}(b). Next, we show that $\Sd_f(\mathcal{G})=\frac{k}{2}$. Let $g:V \rightarrow [0,1]$ be any simultaneous resolving function for $\mathcal{G}$. Since $u_{i+1}$ and $u_{i+2}$ are twins in $G_i$ for each $i\in[k-2]$, $u_k$ and $u_1$ are twins in $G_{k-1}$, and $u_1$ and $u_2$ are twins in $G_k$, we have $g(u_1)+g(u_2) \ge 1$, $g(u_{i+1})+g(u_{i+2})\ge 1$ for each $i\in[k-2]$, and $g(u_k)+g(u_1)\ge 1$. By summing over the $k$ inequalities, we have $2g(V)\ge k$; thus, $\Sd_f(\mathcal{G})\ge \frac{k}{2}$. By Observation~\ref{factor_bounds2}, $\Sd_f(\mathcal{G})\le\frac{k}{2}$. Therefore, $\Sd(\mathcal{G})-\Sd_f(\mathcal{G})=k-1-\frac{k}{2}=\frac{k-2}{2} \rightarrow \infty$ as $k \rightarrow \infty$.
\end{remark}

\section{The simultaneous fractional dimension of a graph and its complement}\label{sec_sdim_frac_complement}

In this section, we focus on the simultaneous fractional dimension of a graph and its complement. For any non-trivial graph $G$, notice that $1\le \Sd_f(G, \overline{G}) \le \frac{|V(G)|}{2}$; we characterize graphs $G$ satisfying $\Sd_f(G, \overline{G})=1$ and $\Sd_f(G, \overline{G})=\frac{|V(G)|}{2}$, respectively. We show that if $\{\diam(G), \diam(\overline{G})\}\neq\{3\}$ and $\diam(G)\le \diam(\overline{G})$, then $\Sd_f(G, \overline{G})=\dim_f(G)$. We also determine $\Sd_f(G, \overline{G})$ when $G$ is a tree or a unicyclic graph.

Based on Observation~\ref{factor_bounds2}, we have the following.

\begin{corollary}\label{bounds_Gbar}
For any graph $G$ of order $n \ge 2$, $$1\le \max\{\dim_f(G), \dim_f(\overline{G})\} \le \Sd_f(G, \overline{G}) \le \min\left\{\dim_f(G)+\dim_f(\overline{G}), \frac{n}{2}\right\} \le \frac{n}{2}.$$
\end{corollary}

Now, for any graph $G$ of order $n \ge 2$, we characterize $G$ satisfying $\Sd_f(G, \overline{G})=1$ and $\Sd_f(G, \overline{G})=\frac{n}{2}$, respectively. For an explicit characterization of connected graphs $G$ with $\dim_f(G)=\frac{|V(G)|}{2}$, we refer to~\cite{jian}. We recall the following results.

\begin{observation}\emph{\cite{SdimGbar}}\label{obs_connected}
Graphs $G$ and $\overline{G}$ cannot both be disconnected.
\end{observation}

\begin{theorem}\emph{\cite{SdimGbar}}\label{global_characterization}
Let $G$ be a graph of order $n \ge 2$. Then $\Sd(G, \overline{G})=1$ if and only if $G \in \{P_2, \overline{P}_2, P_3, \overline{P}_3\}$, and $\Sd(G,\overline{G})=n-1$ if and only if $G \in \{K_n, \overline{K}_n\}$.
\end{theorem}

\begin{theorem}\label{characterization_Gbar}
Let $G$ be a graph of order $n \ge 2$. Then 
\begin{itemize}
\item[(a)] $\Sd_f(G, \overline{G})=1$ if and only if $G\in\{P_2, \overline{P}_2, P_3, \overline{P}_3\}$;
\item[(b)] $\Sd_f(G, \overline{G})=\frac{n}{2}$ if and only if each vertex in $G$ belongs to a twin equivalence class of cardinality at least two in $G$.
\end{itemize}
\end{theorem}

\begin{proof}
Let $G$ be a graph of order $n\ge 2$.

(a) ($\Leftarrow$) Let $G\in\{P_2, \overline{P}_2, P_3, \overline{P}_3\}$. Since $1\le \Sd_f(G, \overline{G})\le \Sd(G, \overline{G})\!=\!1$ by Theorem~\ref{global_characterization}, $\Sd_f(G, \overline{G})\!=\!1$.

($\Rightarrow$) Let $\Sd_f(G, \overline{G})=1$; then $\dim_f(G)=\dim_f(\overline{G})=1$ by Corollary~\ref{bounds_Gbar}. We may, by Observation~\ref{obs_connected}, assume that $G$ is connected; then $G=P_n$ by Theorem~\ref{dim_frac_graph}(a). If $n\ge4$, then $\overline{G}$ is connected but not a path that shares a common end-vertex with $G$; thus, $\Sd_f(G, \overline{G}) \neq 1$ by Theorem~\ref{sdim_frac_path}.

(b) ($\Leftarrow$) Let $Q_1, Q_2, \ldots, Q_x$ be the twin equivalence classes of $G$ such that $|Q_i| \ge 2$ and each vertex in $G$ belongs to some $Q_i$, where $i\in[x]$. Let $g:V(G)\rightarrow [0,1]$ be any resolving function of $G$. Then, for each $i\in[x]$, $g(Q_i) \ge \frac{|Q_i|}{2}$ by Lemma~\ref{obs_twin_frac}. By summing over the $x$ inequalities, we have $g(V(G))=\sum_{i=1}^{x}g(Q_i) \ge \sum_{i=1}^{x} \frac{|Q_i|}{2}=\frac{n}{2}$; thus, $\Sd_f(G,\overline{G})\ge\dim_f(G)\ge \frac{n}{2}$ by Corollary~\ref{bounds_Gbar}. Since $\Sd_f(G, \overline{G})\le \frac{n}{2}$ by Corollary~\ref{bounds_Gbar}, we have $\Sd_f(G,\overline{G})= \frac{n}{2}$.

($\Rightarrow$) Let $\Sd_f(G,\overline{G})= \frac{n}{2}$. Suppose $\{w\}$ is a twin equivalence class of cardinality one in $G$. For $V=V(G)=V(\overline{G})$, let $h:V \rightarrow [0,1]$ be a function defined by $h(w)=0$ and $h(u)=\frac{1}{2}$ for each $u\in V-\{w\}$; we show that $h$ is a resolving function for both $G$ and $\overline{G}$ with $h(V)=\frac{n-1}{2}$. Let $x$ and $y$ be distinct vertices in $V$. First, let $w\not\in \{x,y\}$. Then $R_{G}\{x,y\} \supseteq \{x,y\}$ and $R_{\overline{G}}\{x,y\} \supseteq \{x,y\}$; thus, $h(R_{G}\{x,y\}) \ge h(x)+h(y)=1$ and $h(R_{\overline{G}}\{x,y\}) \ge 1$. Second, let $w\in\{x,y\}$, say $x=w$ without loss of generality. Then there exists $z\in V-\{x,y\}$ such that $d_{G}(z,x)\neq d_{G}(z,y)$; notice that $d_{G}(z,x)= d_{G}(z,y)$, for every $z\in V-\{x,y\}$, means that $x=w$ and $y$ belong to the same twin equivalence class in $G$, which contradicts the assumption that $\{w\}$ is a twin equivalence class of cardinality one in $G$. Similarly, $d_{\overline{G}}(z',x)\neq d_{\overline{G}}(z',y)$ for some  $z'\in V-\{x,y\}$. So, $R_{G}\{x,y\} \supseteq \{x,y,z\}$ and $R_{\overline{G}}\{x,y\} \supseteq \{x,y,z'\}$; thus, $h(R_{G}\{x,y\}) \ge h(x)+h(y)+h(z)=1$ and $h(R_{\overline{G}}\{x,y\}) \ge 1$. Thus, $h$ is a resolving function for both $G$ and $\overline{G}$, and hence $\Sd_f(G, \overline{G}) \le h(V)=\frac{n-1}{2}$. Therefore, each vertex in $G$ must belong to a twin equivalence class of cardinality at least two in $G$ (and thus, in $\overline{G}$).~\hfill
\end{proof}

Next, we show that if $\{\diam(G), \diam(\overline{G})\}\neq\{3\}$ and $\diam(G) \le \diam(\overline{G})$, then $\Sd_f(G, \overline{G})=\dim_f(G)=\max\{\dim_f(G), \dim_f(\overline{G})\}$. We recall the following useful result. 

\begin{proposition}\emph{\cite{BM}}\label{diameter}
If $G$ is a graph with $\diam(G) \ge 4$, then $\diam(\overline{G}) \le 2$.
\end{proposition}

\begin{lemma}\label{main_diam}
Let $G$ be a graph with $\diam(G)=2$. Let $V$ be the common vertex set for $G$ and $\overline{G}$, and let $g:V\rightarrow [0,1]$ be a real-valued function. If $g$ is a minimum resolving function of $G$, then $g$ is a resolving function of $\overline{G}$.
\end{lemma}

\begin{proof}
Let $G$ be a graph with $\diam(G)=2$, and let $V=V(G)=V(\overline{G})$. For any distinct vertices $x,y \in V(G)$, let $z\in R_G\{x,y\}$. Then $d_G(x,z) \neq d_G(y,z)$, say $0\le d_G(x,z) <d_G(y,z) \le 2$ without loss of generality. If $d_G(x,z)=0$ and $d_G(y,z) \in \{1,2\}$, then $d_{\overline{G}}(x,z)=0$ and $d_{\overline{G}}(y,z) \ge 1$. If $d_G(x,z)=1$ and $d_G(y,z)=2$, then $d_{\overline{G}}(x,z) \ge 2$ and $d_{\overline{G}}(y,z)=1$. In each case, $z \in R_{\overline{G}}\{x,y\}$. So, $R_G\{x,y\} \subseteq R_{\overline{G}}\{x,y\}$. If $g: V \rightarrow [0,1]$ is a minimum resolving function of $G$, then $g(R_G\{x,y\}) \ge 1$; thus, $g(R_{\overline{G}}\{x,y\}) \ge 1$. So, $g$ is a resolving function for $\overline{G}$.~\hfill
\end{proof}

\begin{theorem}\label{thm_diam}
Let $G$ be a graph of order at least two. Suppose $\{\diam(G), \diam(\overline{G})\} \neq \{3\}$ and $\diam(G) \le \diam(\overline{G})$. Then $\Sd_f(G, \overline{G})=\dim_f(G)$.
\end{theorem}

\begin{proof}
Let $G$ be a graph of order $n \ge 2$. Then $\diam(G)\le 2$ by the hypotheses and Proposition~\ref{diameter}. If $\diam(G)=1$, then $G=K_n$ and $\Sd_f(G, \overline{G})=\frac{n}{2}=\dim_f(G)=\dim_f(\overline{G})$ by Theorem~\ref{characterization_Gbar}(b). If $\diam(G)=2$, then $\Sd_f(G, \overline{G})=\dim_f(G)=\max\{\dim_f(G), \dim_f(\overline{G})\}$ by Lemma~\ref{main_diam}.~\hfill
\end{proof}

\begin{remark}
Theorem~\ref{thm_diam} implies that $\Sd_f(G, \overline{G})=\dim_f(G)$ when $G$ is the Petersen graph, a complete multi-partite graph, a wheel graph or a fan graph since those graphs have diameter at most two. 
\end{remark}

Next, we show that $\Sd_f(T, \overline{T})=\dim_f(\overline{T})$ for any tree $T \neq P_4$. We first consider $\Sd_f(P_n,\overline{P}_n)$ for $n\ge 2$. 

\begin{proposition}\label{sim_fdim_path}
For $n\ge 2$, 
\begin{equation*}
\Sd_f(P_n, \overline{P}_n)=\left\{
\begin{array}{ll}
\frac{4}{3} & \mbox{ if } n=4,\\
\dim_f(\overline{P}_n) & \mbox{ otherwise }.
\end{array}\right.
\end{equation*}
\end{proposition}

\begin{proof}
If $n\in\{2,3\}$, then $\Sd_f(P_n, \overline{P}_n)=1=\dim_f(P_n)=\dim_f(\overline{P}_n)$ by Corollary~\ref{bounds_Gbar} and Theorem~\ref{characterization_Gbar}(a). If $n\ge 5$, then $\diam(P_n) \ge 4$ and $\diam(\overline{P}_n)=2$; thus, $\Sd_f(P_n, \overline{P}_n)=\dim_f(\overline{P}_n)$ by Theorem~\ref{thm_diam}.

So, suppose $n=4$, and let $G=P_4$ be given by $u_1, u_2, u_3, u_4$; then $\overline{G}=P_4$ is given by $u_3, u_1, u_4, u_2$. Thus, $\Sd_f(P_4, \overline{P}_4) = \frac{4}{3}$ by Corollary~\ref{cor_path_new}.~\hfill
\end{proof}

Now, we recall the following result on trees $T$ with $\diam(T)=3=\diam(\overline{T})$.

\begin{proposition}\emph{\cite{SdimGbar}}\label{tree_diam3}
A tree $T$ of order $n$ satisfies $\diam(T)=3=\diam(\overline{T})$ if and only if $T$ satisfies one of the following three conditions: 
\begin{itemize}
\item[(a)] $T=P_4$;
\item[(b)] $ex(T)=1$ and $\sigma(T)=n-2$;
\item[(c)] $ex(T)=2$ and $\sigma(T)=n-2$.
\end{itemize}
\end{proposition}

\begin{theorem}\label{Sd_tree}
For any non-trivial tree $T \neq P_4$, $\Sd_f(T,\overline{T})=\dim_f(\overline{T})\ge \dim_f(T)$. 
\end{theorem}

\begin{proof}
Let $T$ be a tree of order $n\ge 2$ and let $T \neq P_4$. We consider four cases.

\medskip

\textbf{Case 1: $ex(T)=0$.} In this case, $T=P_n$. By Proposition~\ref{sim_fdim_path} and Corollary~\ref{bounds_Gbar}, $\Sd_f(P_n, \overline{P}_n)=\dim_f(\overline{P}_n)\ge \dim_f(P_n)$ for $n\neq 4$.

\medskip

\textbf{Case 2: $ex(T)=1$.} In this case, $\diam(T)\ge 2$. Let $v$ be the exterior major vertex of $T$. Let $g: V \rightarrow [0,1]$ be any resolving function of $T$, and let $\tilde{g}: V \rightarrow [0,1]$ be any resolving function of $\overline{T}$ on $V=V(T)=V(\overline{T})$. 

First, suppose $\diam(T)=2$. Then $T=K_{1, n-1}$, where $n\ge 4$. Since any two distinct vertices in $N_T(v)=\cup_{i=1}^{n-1}\{\ell_i\}$ are twins in $T$ and $\overline{T}$, we have $\sum_{i=1}^{n-1}g(\ell_i) \ge \frac{n-1}{2}$ and $\sum_{i=1}^{n-1}\tilde{g}(\ell_i) \ge \frac{n-1}{2}$ by Lemma~\ref{obs_twin_frac}. So, $\dim_f(T)\ge \frac{n-1}{2}$, $\dim_f(\overline{T})\ge \frac{n-1}{2}$, and $\Sd_f(T, \overline{T})\ge\frac{n-1}{2}$ by Corollary~\ref{bounds_Gbar}. If $h:V\rightarrow [0,1]$ is a function defined by $h(v)=0$ and $h(\ell_i)=\frac{1}{2}$ for each $i\in[n-1]$, then $h$ is a resolving function for both $T$ and $\overline{T}$ with $h(V)=\frac{n-1}{2}$; thus, $\max\{\dim_f(T),\dim_f(\overline{T}),\Sd_f(T, \overline{T})\}\le \frac{n-1}{2}$. So, $\Sd_f(T, \overline{T})=\frac{n-1}{2}=\dim_f(\overline{T})=\dim_f(T)$.

Second, suppose $\diam(T)=3$; then $T$ is isomorphic to the graph obtained from $K_{1, n-2}$ by subdividing exactly one edge once. By Proposition~\ref{tree_diam3}, $\diam(\overline{T})=3$. Let $\{\ell_1, \ell_2, \ldots, \ell_{n-2}\}$ be the set of terminal vertices of $v$ in $T$, and let $d_T(v, \ell_1)=2$ such that $s$ is the support vertex lying on the $v-\ell_1$ path in $T$. We show that $\Sd_f(T, \overline{T})=\dim_f(\overline{T})=\frac{n-1}{2}=\frac{1}{2}+\dim_f(T)$. Since $R_{\overline{T}}\{s, \ell_1\}=\{v, s, \ell_1\}$ and $R_{\overline{T}}\{\ell_i, \ell_j\}=\{\ell_i, \ell_j\}$ for any distinct $i, j\in\{2,3,\ldots, n-2\}$, we have $\tilde{g}(v)+\tilde{g}(s)+\tilde{g}(\ell_1) \ge 1$ and $\tilde{g}(\cup_{i=2}^{n-2}\{\ell_i\})\ge \frac{n-3}{2}$ by Lemma~\ref{obs_twin_frac}; thus, $\tilde{g}(V)=\tilde{g}(v)+\tilde{g}(s)+\sum_{i=1}^{n-2}\tilde{g}(\ell_i)\ge \frac{n-1}{2}$. So, $\dim_f(\overline{T}) \ge \frac{n-1}{2}$. Since $\dim_f(T)=\frac{n-2}{2}$ by Theorem~\ref{dim_frac_graph}(b), $\Sd_f(T, \overline{T})\ge \dim_f(\overline{T}) \ge \frac{n-1}{2}=\frac{1}{2}+\dim_f(T)$. If we let $h:V\rightarrow [0,1]$ be a function defined by $h(s)=h(\ell_1)=\frac{1}{4}$ and $h(v)=h(\ell_i)=\frac{1}{2}$ for each $i\in\{2,3,\ldots, n-2\}$, then $h$ is a resolving function for both $T$ and $\overline{T}$ with $h(V)=\frac{n-1}{2}$; thus, $\dim_f(\overline{T}) \le \frac{n-1}{2}$ and $\Sd_f(T, \overline{T}) \le \frac{n-1}{2}$. Therefore, $\Sd_f(T, \overline{T})=\dim_f(\overline{T})=\frac{n-1}{2}=\frac{1}{2}+\dim_f(T)$.

Third, suppose $\diam(T) \ge 4$. Then $\diam(\overline{T})\le 2$ by Proposition~\ref{diameter}, and $\Sd_f(T, \overline{T})=\dim_f(\overline{T})$ by Theorem~\ref{thm_diam}.

\medskip

\textbf{Case 3: $ex(T)=2$.} In this case, $\diam(T)\ge 3$. First, suppose $\diam(T)=3$; then $\sigma(T)=n-2$. By Proposition~\ref{tree_diam3}, $\diam(\overline{T})=3$. Let $v_1$ and $v_2$ be the two distinct exterior major vertices of $T$ with $ter_T(v_1)=a \ge 2$ and $ter_T(v_2)=b \ge 2$. Let $\{\ell_1, \ell_2, \ldots, \ell_a\}$ be the set of terminal vertices of $v_1$ in $T$, and let $\{\ell'_1, \ell'_2, \ldots, \ell'_b\}$ be the set of terminal vertices of $v_2$ in $T$. Let $g: V \rightarrow [0,1]$ be a simultaneous resolving function of the family $\{T, \overline{T}\}$ on the common vertex set $V$. Since any distinct vertices $\ell_i$ and $\ell_j$ are twins in $T$ and $\overline{T}$ for $i,j\in[a]$, we have $\sum_{i=1}^{a}g(\ell_i) \ge \frac{a}{2}$ by Lemma~\ref{obs_twin_frac}; similarly, noting that any distinct vertices $\ell'_i$ and $\ell'_j$ are twins in $T$ and $\overline{T}$ for $i,j\in[b]$, we have $\sum_{j=1}^{b}g(\ell'_j) \ge \frac{b}{2}$ by Lemma~\ref{obs_twin_frac}. So, $g(V) \ge (\sum_{i=1}^{a}g(\ell_i))+(\sum_{j=1}^{b}g(\ell'_j))\ge\frac{a+b}{2}=\frac{n-2}{2}$ and $\min\{\dim_f(T), \dim_f(\overline{T}),\Sd_f(T, \overline{T})\} \ge \frac{n-2}{2}$. If $h:V\rightarrow [0,1]$ is a function defined by $h(v_1)=h(v_2)=0$ and $h(\ell_i)=h(\ell'_j)=\frac{1}{2}$ for each $i\in[a]$ and each $j\in[b]$, then $h$ is a resolving function for both $T$ and $\overline{T}$ with $h(V)=\frac{n-2}{2}$. So, $\max\{\dim_f(T),\dim_f(\overline{T}),\Sd_f(T, \overline{T})\} \le \frac{n-2}{2}$. Thus, $\Sd_f(T, \overline{T})=\dim_f(\overline{T})=\dim_f(T)=\frac{n-2}{2}$. Next, suppose $\diam(T) \ge 4$. Then $\diam(\overline{T})\le 2$ by Proposition~\ref{diameter}, and $\Sd_f(T, \overline{T})=\dim_f(\overline{T})$ by Theorem~\ref{thm_diam}.

\medskip

\textbf{Case 4: $ex(T)\ge 3$.} In this case, $\diam(T)\ge 4$. Since $\diam(\overline{T})\le 2$ by Proposition~\ref{diameter}, $\Sd_f(T, \overline{T})=\dim_f(\overline{T})$ by Theorem~\ref{thm_diam}.~\hfill
\end{proof}

Next, we determine $\Sd_f(G, \overline{G})$ when $G$ is a unicyclic graph; we show that if $G$ is a unicyclic graph of order at least seven, then $\Sd_f(G, \overline{G})=\dim_f(\overline{G})$. We first consider $\Sd_f(C_n, \overline{C}_n)$ for $n\ge 3$.

\begin{proposition}\label{sdim_frac_cycle}
For $n \ge 3$, $\Sd_f(C_n, \overline{C}_n)=\dim_f(\overline{C}_n)=\left\{
\begin{array}{ll}
\frac{n}{2} & \mbox{ if }n\in\{3,4\},\\
\frac{n}{4} & \mbox{ if } n\ge 5.
\end{array}\right.
$.
\end{proposition}

\begin{proof}
If $n=3$, then $V(C_3)$ is a twin equivalence  class of cardinality three in $C_3$, and $V(\overline{C}_3)$ is a twin equivalence  class of cardinality three in $\overline{C}_3$; thus, $\Sd_f(C_3, \overline{C}_3)=\frac{3}{2}=\dim_f(C_3)=\dim_f(\overline{C}_3)$ by Theorem~\ref{characterization_Gbar}(b), Lemma~\ref{obs_twin_frac} and Corollary~\ref{bounds_Gbar}. If $n=4$, then $\dim_f(C_4)=2$ by Theorem~\ref{dim_frac_graph}(c), and $\dim_f(\overline{C}_4)=\dim_f(2K_2)=2$ by Theorem~\ref{dim_frac_graph}(a); thus, $\Sd_f(C_4, \overline{C}_4)=2=\dim_f(C_4)=\dim_f(\overline{C}_4)$ by Corollary~\ref{bounds_Gbar}. If $n=5$, then $\overline{C}_5$ is isomorphic to $C_5$; thus, $\Sd_f(C_5,\overline{C}_5)=\frac{5}{4}=\dim_f(C_5)=\dim_f(\overline{C}_5)$ by Theorem~\ref{dim_frac_graph}(c) and Proposition~\ref{sdim_frac_vtransitive}. So, suppose $n \ge 6$; then $\diam(\overline{C}_n)=2$. So, $\Sd_f(C_n, \overline{C}_n)=\dim_f(\overline{C}_n)$ by Theorem~\ref{thm_diam}. Since $\overline{C}_n$ is a vertex-transitive graph and $\min\{|R_{\overline{C}_n}\{x,y\}|:x,y\in V(\overline{C}_n) \mbox{ and }x\neq y\}=4$, $\dim_f(\overline{C}_n)=\frac{n}{4}$ by Theorem~\ref{Feng_vtransitive}; thus, $\Sd_f(C_n, \overline{C}_n)=\dim_f(\overline{C}_n)=\frac{n}{4}$ for $n\ge 6$.~\hfill
\end{proof}

Now, we recall the unicyclic graphs $G$ with $\diam(G)=3=\diam(\overline{G})$.

\begin{proposition}\emph{\cite{SdimGbar}}\label{unicyclic_diam3}
Let $G$ be a unicyclic graph of order $n$ with $\diam(G)=3=\diam(\overline{G})$. Then $G$ is isomorphic to one of the graphs in Figure~\ref{fig_unicyclic33}. 
\end{proposition}

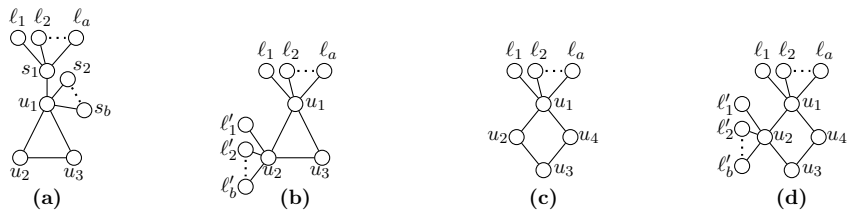
\begin{figure}[ht]
\centering
\begin{tikzpicture}[scale=.55, transform shape]

\node [draw, shape=circle, scale=1] (a0) at  (0,1.3) {};
\node [draw, shape=circle, scale=1] (a1) at  (-0.65,0) {};
\node [draw, shape=circle, scale=1] (a2) at  (0.65,0) {};
\node [draw, shape=circle, scale=1] (a3) at  (0,2.1) {};
\node [draw, shape=circle, scale=1] (a4) at  (-0.2,2.9) {};
\node [draw, shape=circle, scale=1] (a31) at  (0.5,1.9) {};
\node [draw, shape=circle, scale=1] (a32) at  (0.9,1.15) {};
\node [draw, shape=circle, scale=1] (a41) at  (-0.7,2.9) {};
\node [draw, shape=circle, scale=1] (a42) at  (0.7,2.9) {};

\draw(a0)--(a1)--(a2)--(a0);\draw(a0)--(a3)--(a4);\draw(a31)--(a0)--(a32);
\draw(a41)--(a3)--(a42);\draw[thick, dotted](0.05,2.9)--(0.45,2.9);\draw[thick, dotted](0.6,1.7)--(0.8,1.35);

\node [scale=1.3] at (-0.7,3.4) {$\ell_1$};
\node [scale=1.3] at (-0.1,3.4) {$\ell_2$};
\node [scale=1.3] at (0.85,3.4) {$\ell_a$};
\node [scale=1.3] at (-0.4,2.1) {$s_1$};
\node [scale=1.3] at (0.85,2.2) {$s_2$};
\node [scale=1.3] at (1.35,1.15) {$s_b$};
\node [scale=1.3] at (-0.45,1.29) {$u_1$};
\node [scale=1.3] at (0.65,-0.4) {$u_3$};
\node [scale=1.3] at (-0.65,-0.4) {$u_2$};

\node [draw, shape=circle, scale=1] (b0) at  (6,1.3) {};
\node [draw, shape=circle, scale=1] (b1) at  (5.35,0) {};
\node [draw, shape=circle, scale=1] (b2) at  (6.65,0) {};
\node [draw, shape=circle, scale=1] (b11) at  (5.3,2.1) {};
\node [draw, shape=circle, scale=1] (b12) at  (5.8,2.1) {};
\node [draw, shape=circle, scale=1] (b13) at  (6.7,2.1) {};
\node [draw, shape=circle, scale=1] (b21) at  (4.8,0.8) {};
\node [draw, shape=circle, scale=1] (b22) at  (4.8,0.2) {};
\node [draw, shape=circle, scale=1] (b23) at  (4.8,-0.7) {};

\draw(b0)--(b1)--(b2)--(b0);\draw(b11)--(b0)--(b12);\draw(b0)--(b13);
\draw(b21)--(b1)--(b22);\draw(b1)--(b23);
\draw[thick, dotted](6.05,2.1)--(6.45,2.1);\draw[thick, dotted](4.8,-0.05)--(4.8,-0.45);

\node [scale=1.3] at (5.3,2.6) {$\ell_1$};
\node [scale=1.3] at (5.9,2.6) {$\ell_2$};
\node [scale=1.3] at (6.8,2.6) {$\ell_a$};
\node [scale=1.3] at (4.4,0.8) {$\ell'_1$};
\node [scale=1.3] at (4.4,0.2) {$\ell'_2$};
\node [scale=1.3] at (4.4,-0.7) {$\ell'_b$};
\node [scale=1.3] at (6.5,1.3) {$u_1$};
\node [scale=1.3] at (6.6,-0.35) {$u_3$};
\node [scale=1.3] at (5.45,-0.35) {$u_2$};

\node [draw, shape=circle, scale=1] (d0) at  (12,1.3) {};
\node [draw, shape=circle, scale=1] (d1) at  (11.35,0.5) {};
\node [draw, shape=circle, scale=1] (d2) at  (12.65,0.5) {};
\node [draw, shape=circle, scale=1] (d3) at  (12,-0.3) {};
\node [draw, shape=circle, scale=1] (d11) at  (11.3,2.1) {};
\node [draw, shape=circle, scale=1] (d12) at  (11.8,2.1) {};
\node [draw, shape=circle, scale=1] (d13) at  (12.7,2.1) {};

\draw(d0)--(d1)--(d3)--(d2)--(d0);\draw(d11)--(d0)--(d12);\draw(d0)--(d13);\draw[thick, dotted](12.05,2.1)--(12.45,2.1);

\node [scale=1.3] at (11.3,2.6) {$\ell_1$};
\node [scale=1.3] at (11.9,2.6) {$\ell_2$};
\node [scale=1.3] at (12.8,2.6) {$\ell_a$};
\node [scale=1.3] at (12.5,1.3) {$u_1$};
\node [scale=1.3] at (10.9,0.5) {$u_2$};\node [scale=1.3] at (13.1,0.5) {$u_4$};\node [scale=1.3] at (12.5,-0.3) {$u_3$};

\node [draw, shape=circle, scale=1] (e0) at  (18,1.3) {};
\node [draw, shape=circle, scale=1] (e1) at  (17.35,0.5) {};
\node [draw, shape=circle, scale=1] (e2) at  (18.65,0.5) {};
\node [draw, shape=circle, scale=1] (e3) at  (18,-0.3) {};
\node [draw, shape=circle, scale=1] (e11) at  (17.3,2.1) {};
\node [draw, shape=circle, scale=1] (e12) at  (17.8,2.1) {};
\node [draw, shape=circle, scale=1] (e13) at  (18.7,2.1) {};
\node [draw, shape=circle, scale=1] (e21) at  (16.8,1.3) {};
\node [draw, shape=circle, scale=1] (e22) at  (16.8,0.7) {};
\node [draw, shape=circle, scale=1] (e23) at  (16.8,-0.2) {};

\draw(e0)--(e1)--(e3)--(e2)--(e0);\draw(e11)--(e0)--(e12);\draw(e0)--(e13);\draw[thick, dotted](18.05,2.1)--(18.45,2.1);
\draw(e21)--(e1)--(e22);\draw(e1)--(e23);\draw[thick, dotted](16.8,0.45)--(16.8,0.05);

\node [scale=1.3] at (17.3,2.6) {$\ell_1$};
\node [scale=1.3] at (17.9,2.6) {$\ell_2$};
\node [scale=1.3] at (18.8,2.6) {$\ell_a$};
\node [scale=1.3] at (16.4,1.3) {$\ell'_1$};
\node [scale=1.3] at (16.4,0.7) {$\ell'_2$};
\node [scale=1.3] at (16.4,-0.2) {$\ell'_b$};
\node [scale=1.3] at (18.5,1.3) {$u_1$};
\node [scale=1.3] at (17.8,0.5) {$u_2$};\node [scale=1.3] at (19.1,0.5) {$u_4$};\node [scale=1.3] at (18.5,-0.3) {$u_3$};

\node [scale=1.3] at (0,-1) {\textbf{(a)}};
\node [scale=1.3] at (6,-1) {\textbf{(b)}};
\node [scale=1.3] at (12,-1) {\textbf{(c)}};
\node [scale=1.3] at (18,-1) {\textbf{(d)}};

\end{tikzpicture}
\caption{Unicyclic graphs $G$ with $\diam(G)=3=\diam(\overline{G})$, where $a,b \ge 1$.}\label{fig_unicyclic33}
\end{figure}

\begin{proposition}\label{sdim_frac_unicyclic_main}
Let $G$ be a unicyclic graph with $\diam(G)=3=\diam(\overline{G})$. Then 
\begin{equation*}
\Sd_f(G, \overline{G})=\left\{
\begin{array}{ll}
\dim_f(G)=2=\frac{1}{2}+\dim_f(\overline{G}) & \mbox{if $G$ is isomorphic to Figure~\ref{fig_unicyclic33}(c) with $a=1$},\\
\dim_f(G)=2=\frac{1}{3}+\dim_f(\overline{G}) & \mbox{if $G$ is isomorphic to Figure~\ref{fig_unicyclic33}(d) with $a=b=1$},\\
\dim_f(\overline{G}) & \mbox{otherwise}.
\end{array}\right.
\end{equation*}
\end{proposition}

\begin{proof}
Let $G$ be a unicyclic graph with $\diam(G)=3=\diam(\overline{G})$. By Proposition~\ref{unicyclic_diam3}, $G$ is isomorphic to one of the graphs in Figure~\ref{fig_unicyclic33}. Let $g: V \rightarrow [0,1]$ be any resolving function of $G$, and let $\tilde{g}: V \rightarrow [0,1]$ be any resolving function of $\overline{G}$ on the common vertex set $V=V(G)=V(\overline{G})$. 

\medskip

\textbf{Case 1: $G$ is isomorphic to Figure~\ref{fig_unicyclic33}(a).} In this case, we show that $\Sd_f(G, \overline{G})=\dim_f(\overline{G})$.

First, suppose $a=1$ and $b\in\{1,2\}$. We show that $\dim_f(\overline{G})=2$. Since $R_{\overline{G}}\{u_2, u_3\}=\{u_2, u_3\}$ and $R_{\overline{G}}\{\ell_1, s_1\}=\{\ell_1, s_1, u_1\}$, we have $\tilde{g}(u_2)+\tilde{g}(u_3) \ge1$ and $\tilde{g}(\ell_1)+\tilde{g}(s_1)+\tilde{g}(u_1)\ge 1$; thus, $\tilde{g}(V)\ge 2$, which implies $\dim_f(\overline{G})\ge 2$. If we let $h: V \rightarrow [0,1]$ be a function defined by $h(\ell_1)=h(s_1)=h(u_2)=h(u_3)=\frac{1}{2}$ and $h(v)=0$ for each $v\in V-\{\ell_1, s_1,u_2,u_3\}$, then $h$ is a resolving function for both $G$ and $\overline{G}$ with $h(V)=2$; thus, $\max\{\dim_f(G),\dim_f(\overline{G}),\Sd_f(G, \overline{G})\}\le 2$. By Corollary~\ref{bounds_Gbar}, $\Sd_f(G, \overline{G})=\dim_f(\overline{G})=2$.

Second, suppose $a=1$ and $b\ge3$. We note the following: (i) $\tilde{g}(\ell_1)+\tilde{g}(s_1)+\tilde{g}(u_1)\ge 1$ since $R_{\overline{G}}\{\ell_1, s_1\}=\{\ell_1, s_1, u_1\}$; (ii) $\tilde{g}(u_2)+\tilde{g}(u_3)\ge 1$ since $u_2$ and $u_3$ are twins in $\overline{G}$; (iii) $\sum_{i=2}^{b}\tilde{g}(s_i)\ge\frac{b-1}{2}$ by Lemma~\ref{obs_twin_frac}. So, $\tilde{g}(V) \ge \frac{b+3}{2}$, and hence $\dim_f(\overline{G})\ge \frac{b+3}{2}$. If we let $h:V \rightarrow [0,1]$ be a function defined by $h(u_1)=0$ and $h(v)=\frac{1}{2}$ for each $v\in V-\{u_1\}$, then $h$ is a resolving function for both $G$ and $\overline{G}$ with $h(V)=\frac{b+3}{2}$; thus, $\max\{\dim_f(G),\dim_f(\overline{G}),\Sd_f(G, \overline{G})\} \le \frac{b+3}{2}$. By Corollary~\ref{bounds_Gbar}, $\Sd_f(G, \overline{G})=\dim_f(\overline{G})=\frac{b+3}{2}$.

Third, suppose $a\ge 2$ and $b\in\{1,2\}$. We show that $\dim_f(\overline{G})=\frac{a+3}{2}$. If $b=1$, then $R_{\overline{G}}\{u_1, u_2\}=\{s_1, u_1, u_2\}$, $R_{\overline{G}}\{u_1, u_3\}=\{s_1, u_1, u_3\}$ and $R_{\overline{G}}\{u_2, u_3\}=\{u_2, u_3\}$; thus, $\tilde{g}(s_1)+\tilde{g}(u_1)+\tilde{g}(u_2)\ge 1$, $\tilde{g}(s_1)+\tilde{g}(u_1)+\tilde{g}(u_3)\ge 1$ and $\tilde{g}(u_2)+\tilde{g}(u_3)\ge 1$. If $b=2$, then $R_{\overline{G}}\{u_1, u_2\}=\{s_1,s_2, u_1, u_2\}$, $R_{\overline{G}}\{u_1, u_3\}=\{s_1, s_2,u_1, u_3\}$ and $R_{\overline{G}}\{u_2, u_3\}=\{u_2, u_3\}$; thus, $\tilde{g}(s_1)+\tilde{g}(s_2)+\tilde{g}(u_1)+\tilde{g}(u_2)\ge 1$, $\tilde{g}(s_1)+\tilde{g}(s_2)+\tilde{g}(u_1)+\tilde{g}(u_3)\ge 1$ and $\tilde{g}(u_2)+\tilde{g}(u_3)\ge 1$. In each case, by summing over the three inequalities, we have $\tilde{g}(V-(\cup_{i=1}^{a}\{\ell_i\})) \ge \frac{3}{2}$. Since $\tilde{g}(\cup_{i=1}^{a}\{\ell_i\})\ge \frac{a}{2}$ by Lemma~\ref{obs_twin_frac}, we have $\tilde{g}(V) \ge \frac{a+3}{2}$, and thus $\dim_f(\overline{G})\ge \frac{a+3}{2}$. If we let $h: V \rightarrow [0,1]$ be a function defined by $h(\ell_i)=h(s_1)=h(u_2)=h(u_3)=\frac{1}{2}$ for each $i\in[a]$ and $h(v)=0$ for each $v\in V-((\cup_{i=1}^{a}\{\ell_i\}) \cup \{s_1, u_2, u_3\})$, then $h$ is a resolving function for both $G$ and $\overline{G}$ with $h(V)=\frac{a+3}{2}$; thus, $\max\{\dim_f(G),\dim_f(\overline{G}),\Sd_f(G, \overline{G})\}\le \frac{a+3}{2}$. By Corollary~\ref{bounds_Gbar}, $\Sd_f(G, \overline{G})=\dim_f(\overline{G})=\frac{a+3}{2}$.

Fourth, suppose $a\ge 2$ and $b\ge3$. By Lemma~\ref{obs_twin_frac}, $\sum_{i=1}^{a}\tilde{g}(\ell_i)\ge \frac{a}{2}$, $\sum_{i=2}^{b}\tilde{g}(s_i)\ge \frac{b-1}{2}$, and $\tilde{g}(u_2)+\tilde{g}(u_3)\ge 1$. By summing over the three inequalities, we have $\tilde{g}(V) \ge \frac{a+b+1}{2}$, and thus $\dim_f(\overline{G})\ge \frac{a+b+1}{2}$. If we let $h:V\rightarrow [0,1]$ be a function defined by $h(s_1)=h(u_1)=0$ and $h(v)=\frac{1}{2}$ for each $v\in V-\{s_1, u_1\}$, then $h$ is a resolving function for both $G$ and $\overline{G}$ with $h(V)=\frac{a+b+1}{2}$; thus, $\max\{\dim_f(G),\dim_f(\overline{G}),\Sd_f(G, \overline{G})\} \le \frac{a+b+1}{2}$. By Corollary~\ref{bounds_Gbar}, $\Sd_f(G, \overline{G})=\dim_f(\overline{G})=\frac{a+b+1}{2}$.

\medskip

\textbf{Case 2: $G$ is isomorphic to Figure~\ref{fig_unicyclic33}(b).} In this case, we show that $\Sd_f(G, \overline{G})=\dim_f(\overline{G})$.

First, suppose $a=b=1$. We show that $\dim_f(\overline{G})=\frac{3}{2}$. Since $R_{\overline{G}}\{u_1, u_2\}=V-\{u_3\}$, $R_{\overline{G}}\{\ell_1, u_3\}=V-\{\ell'_1,u_1\}$, and $R_{\overline{G}}\{\ell'_1, u_3\}=V-\{\ell_1, u_2\}$, we have $\tilde{g}(V)-\tilde{g}(u_3)\ge1$, $\tilde{g}(V)-(\tilde{g}(\ell'_1)+\tilde{g}(u_1))\ge 1$, and $\tilde{g}(V)-(\tilde{g}(\ell_1)+\tilde{g}(u_2)) \ge 1$. By summing over the three inequalities, we have $2\tilde{g}(V)\ge 3$, i.e., $\tilde{g}(V) \ge \frac{3}{2}$; thus, $\dim_f(\overline{G})\ge \frac{3}{2}$. If we let $h: V \rightarrow [0,1]$ be a function defined by $h(\ell_1)=h(\ell'_1)=0$ and $h(u_i)=\frac{1}{2}$ for each $i \in[3]$, then $h$ forms a resolving function for both $G$ and $\overline{G}$ with $h(V)=\frac{3}{2}$; thus, $\max\{\dim_f(G),\dim_f(\overline{G}),\Sd_f(G, \overline{G})\}\le \frac{3}{2}$. By Corollary~\ref{bounds_Gbar}, $\Sd_f(G, \overline{G})=\dim_f(\overline{G})=\frac{3}{2}$.

Second, suppose $a\ge 2=1+b$ or $b\ge 2=1+a$, say the former. We note the following: (i) $\tilde{g}(\ell'_1)+\tilde{g}(u_1)+\tilde{g}(u_3) \ge 1$ since $R_{\overline{G}}\{\ell'_1, u_3\}=\{\ell'_1, u_1, u_3\}$; (ii) $\sum_{i=1}^{a}\tilde{g}(\ell_i) \ge \frac{a}{2}$ by  Lemma~\ref{obs_twin_frac}. So, $\tilde{g}(V) \ge \frac{a+2}{2}$, and hence $\dim_f(\overline{G})\ge \frac{a+2}{2}$. If we let $h:V\rightarrow [0,1]$ be a function defined by $h(u_1)=h(u_2)=0$ and $h(v)=\frac{1}{2}$ for each $v\in V-\{u_1, u_2\}$, then $h$ forms a resolving function for both $G$ and $\overline{G}$ with $h(V)=\frac{a+2}{2}$; thus, $\max\{\dim_f(G),\dim_f(\overline{G}),\Sd_f(G, \overline{G})\}  \le \frac{a+2}{2}$. By Corollary~\ref{bounds_Gbar}, $\Sd_f(G, \overline{G})=\dim_f(\overline{G})=\frac{a+2}{2}$.  

Third, suppose $a\ge 2$ and $b \ge 2$. We show that $\dim_f(\overline{G})=\frac{a+b+1}{2}$. We note the following: (i) $\tilde{g}(\cup_{i=1}^{a}\{\ell_i\}) \ge\frac{a}{2}$ by Lemma~\ref{obs_twin_frac}; (ii) for each $i\in[a]$, $\tilde{g}(\ell_i)+\tilde{g}(u_2)+\tilde{g}(u_3)\ge 1$ since $R_{\overline{G}}\{\ell_i, u_3\}=\{\ell_i, u_2, u_3\}$; (iii) $\tilde{g}(\cup_{i=1}^{b}\{\ell'_i\}) \ge\frac{b}{2}$ by Lemma~\ref{obs_twin_frac}; (iv) for each $j\in[b]$, $\tilde{g}(\ell'_j)+\tilde{g}(u_1)+\tilde{g}(u_3)\ge 1$ since $R_{\overline{G}}\{\ell'_j, u_3\}=\{\ell'_j, u_1, u_3\}$; (v) for each $i\in[a]$ and for each $j\in[b]$, $\tilde{g}(\ell_i)+\tilde{g}(\ell'_j)+\tilde{g}(u_1)+\tilde{g}(u_2)\ge 1$ since $R_{\overline{G}}\{\ell_i, \ell'_j\}=\{\ell_i, \ell'_j, u_1, u_2\}$. If $a=b=2$, then we have the following:  $2(\tilde{g}(\ell_1)+\tilde{g}(\ell_2))\ge 2$ from (i) ; $\tilde{g}(\ell_1)+\tilde{g}(u_2)+\tilde{g}(u_3)\ge 1$ and $\tilde{g}(\ell_2)+\tilde{g}(u_2)+\tilde{g}(u_3)\ge 1$ from (ii);  $2(\tilde{g}(\ell'_1)+\tilde{g}(\ell'_2))\ge 2$ from (iii); $\tilde{g}(\ell'_1)+\tilde{g}(u_1)+\tilde{g}(u_3)\ge 1$ and $\tilde{g}(\ell'_2)+\tilde{g}(u_1)+\tilde{g}(u_3)\ge 1$ from (iv); $\tilde{g}(\ell_1)+\tilde{g}(\ell'_1)+\tilde{g}(u_1)+\tilde{g}(u_2)\ge 1$ and $\tilde{g}(\ell_2)+\tilde{g}(\ell'_2)+\tilde{g}(u_1)+\tilde{g}(u_2)\ge 1$ from (v). By summing over the above inequalities, we have $4\tilde{g}(V)\ge10$, i.e., $\tilde{g}(V)\ge \frac{5}{2}$; thus, $\dim_f(\overline{G}) \ge \frac{5}{2}=\frac{a+b+1}{2}$. If $a\ge 3=1+b$ or $b\ge 3=1+a$, say the former, then we have the following: $4(\sum_{i=2}^{a}\tilde{g}(\ell_i))\ge 2(a-1)$ and $2(\tilde{g}(\ell'_1)+\tilde{g}(\ell'_2))\ge 2$ by Lemma~\ref{obs_twin_frac}; $2(\tilde{g}(\ell_1)+\tilde{g}(u_2)+\tilde{g}(u_3))\ge 2$ from (ii); $\tilde{g}(\ell'_1)+\tilde{g}(u_1)+\tilde{g}(u_3)\ge 1$ and  $\tilde{g}(\ell'_2)+\tilde{g}(u_1)+\tilde{g}(u_3)\ge 1$  from (iv); $\tilde{g}(\ell_1)+\tilde{g}(\ell'_1)+\tilde{g}(u_1)+\tilde{g}(u_2)\ge 1$ and $\tilde{g}(\ell_1)+\tilde{g}(\ell'_2)+\tilde{g}(u_1)+\tilde{g}(u_2)\ge 1$ from (v). By summing over the above inequalities, we have $4\tilde{g}(V)\ge 2a+6$, i.e., $\tilde{g}(V)\ge \frac{a+3}{2}$; thus, $\dim_f(\overline{G}) \ge \frac{a+3}{2}=\frac{a+b+1}{2}$. If $a\ge 3$ and $b\ge 3$, then we have the following: $2(\sum_{i=2}^{a}\tilde{g}(\ell_i))\ge a-1$ and $2(\sum_{i=2}^{b}\tilde{g}(\ell'_i))\ge b-1$ by Lemma~\ref{obs_twin_frac}; $\tilde{g}(\ell_1)+\tilde{g}(u_2)+\tilde{g}(u_3)\ge 1$ from (ii); $\tilde{g}(\ell'_1)+\tilde{g}(u_1)+\tilde{g}(u_3)\ge 1$ from (iv); $\tilde{g}(\ell_1)+\tilde{g}(\ell'_1)+\tilde{g}(u_1)+\tilde{g}(u_2)\ge 1$ from (v). By summing over the five inequalities above, we have $2\tilde{g}(V) \ge a+b+1$, and thus $\dim_f(\overline{G})\ge \frac{a+b+1}{2}$. Now, if we let $h: V\rightarrow [0,1]$ be a function defined by $h(u_1)=h(u_2)=0$ and $h(v)=\frac{1}{2}$ for each $v\in V-\{u_1, u_2\}$, then $h$ is a resolving function for both $G$ and $\overline{G}$ with $h(V)=\frac{a+b+1}{2}$; thus, $\max\{\dim_f(G),\dim_f(\overline{G}),\Sd_f(G, \overline{G})\}\le \frac{a+b+1}{2}$. By Corollary~\ref{bounds_Gbar}, $\Sd_f(G, \overline{G})=\dim_f(\overline{G})=\frac{a+b+1}{2}$.  

\medskip

\textbf{Case 3: $G$ is isomorphic to Figure~\ref{fig_unicyclic33}(c).} In this case, we show that $\Sd_f(G, \overline{G})=\dim_f(G)=2=\frac{1}{2}+\dim_f(\overline{G})$ for $a=1$, and $\Sd_f(G, \overline{G})=\dim_f(G)=\frac{a+2}{2}=\dim_f(\overline{G})$ for $a\ge2$.

First, suppose $a=1$.  We show that $\dim_f(G)=2=\frac{1}{2}+\dim_f(\overline{G})$. Since $R_{\overline{G}}\{\ell_1, u_2\}=V-\{u_4\}$, $R_{\overline{G}}\{\ell_1, u_4\}=V-\{u_2\}$ and $R_{\overline{G}}\{u_2, u_4\}=\{u_2,u_4\}$, we have $\tilde{g}(V)-\tilde{g}(u_4)\ge 1$, $\tilde{g}(V)-\tilde{g}(u_2)\ge 1$ and $\tilde{g}(u_2)+\tilde{g}(u_4)\ge 1$. By summing over the three inequalities, we have $2\tilde{g}(V)\ge 3$, and thus $\dim_f(\overline{G})\ge \frac{3}{2}$. If $h': V \rightarrow [0,1]$ is a function defined by $h'(\ell_1)=h'(u_3)=0$ and $h'(v)=\frac{1}{2}$ for each $V-\{\ell_1, u_3\}$, then $h'$ is a resolving function for $\overline{G}$ with $h'(V)=\frac{3}{2}$; thus, $\dim_f(\overline{G})\le \frac{3}{2}$. Now, noting that $R_G\{u_2,u_4\}=\{u_2, u_4\}$ and $R_G\{u_1, u_3\}=\{\ell_1,u_1, u_3\}$, we have $g(u_2)+g(u_4) \ge 1$ and $g(\ell_1)+g(u_1) +g(u_3)\ge 1$; thus, $g(V) \ge 2$ and hence $\dim_f(G) \ge 2$. If $h: V \rightarrow [0,1]$ is a function defined by $h(\ell_1)=0$ and $h(v)=\frac{1}{2}$ for each $v\in V-\{\ell_1\}$, then $h$ is a resolving function for both $G$ and $\overline{G}$ with $h(V)=2$; thus, $\dim_f(G)\le 2$. By Corollary~\ref{bounds_Gbar}, $\Sd_f(G, \overline{G})=\dim_f(G)=2=\frac{1}{2}+\dim_f(\overline{G})$.

Next, suppose $a \ge 2$. We show that $\dim_f(G)=\frac{a+2}{2}=\dim_f(\overline{G})$. By Lemma~\ref{obs_twin_frac}, ${g}(\cup_{i=1}^{a}\{\ell_i\})\ge \frac{a}{2}$ and $g(u_2)+g(u_4)\ge 1$; thus, $g(V) \ge \frac{a+2}{2}$ and hence $\dim_f(G) \ge \frac{a+2}{2}$. The same argument works to show $\dim_f(\overline{G})\ge \frac{a+2}{2}$. If we let $h:V \rightarrow [0,1]$ be a function defined by $h(u_1)=h(u_3)=0$ and $h(v)=\frac{1}{2}$ for each $v\in V-\{u_1, u_3\}$, then $h$ is a resolving function for both $G$ and $\overline{G}$ with $h(V)=\frac{a+2}{2}$; thus, $\max\{\dim_f(G),\dim_f(\overline{G}),\Sd_f(G, \overline{G})\} \le \frac{a+2}{2}$. By Corollary~\ref{bounds_Gbar}, $\Sd_f(G, \overline{G})=\dim_f(G)=\dim_f(\overline{G})=\frac{a+2}{2}$.  

\medskip

\textbf{Case 4: $G$ is isomorphic to Figure~\ref{fig_unicyclic33}(d).} In this case, we show that $\Sd_f(G, \overline{G})=\dim_f(G)=2=\frac{1}{3}+\dim_f(\overline{G})$ when $a=b=1$, and $\Sd_f(G, \overline{G})=\dim_f(\overline{G})$ when $a\ge 2$ or $b\ge 2$.

First, suppose $a=b=1$. We show that $\Sd_f(G, \overline{G})=\dim_f(G)=2=\frac{1}{3}+\dim_f(\overline{G})$. Since $R_{\overline{G}}\{\ell_1, u_1\}=\{\ell_1, u_1, u_2, u_4\}$, $R_{\overline{G}}\{\ell'_1,u_2\}=\{\ell'_1, u_1, u_2, u_3\}$, $R_{\overline{G}}\{\ell_1, u_4\}=\{\ell_1, u_3, u_4\}$, $R_{\overline{G}}\{\ell'_1, u_3\}=\{\ell'_1, u_3, u_4\}$ and $R_{\overline{G}}\{\ell_1, \ell'_1\}=\{\ell_1, \ell'_1, u_1, u_2\}$, we have $\tilde{g}(\ell_1)+\tilde{g}(u_1)+\tilde{g}(u_2)+\tilde{g}(u_4)\ge 1$, $\tilde{g}(\ell'_1)+\tilde{g}(u_1)+\tilde{g}(u_2)+\tilde{g}(u_3)\ge 1$, $\tilde{g}(\ell_1)+\tilde{g}(u_3)+\tilde{g}(u_4)\ge 1$, $\tilde{g}(\ell'_1)+\tilde{g}(u_3)+\tilde{g}(u_4)\ge 1$ and $\tilde{g}(\ell_1)+\tilde{g}(\ell'_1)+\tilde{g}(u_1)+\tilde{g}(u_2)\ge 1$. By summing over the five inequalities, we have $3\tilde{g}(V)\ge 5$, and thus $\dim_f(\overline{G})\ge \frac{5}{3}$. If $h':V \rightarrow [0,1]$ is a function defined by $h'(u_1)=h'(u_2)=\frac{1}{6}$ and $h'(v)=\frac{1}{3}$ for each $V-\{u_1, u_2\}$, then $h'$ is a resolving function for $\overline{G}$ with $h'(V)=\frac{5}{3}$; thus, $\dim_f(\overline{G}) \le \frac{5}{3}$. Now, noting that $R_G\{u_1, u_3\}=\{\ell_1, u_1, u_3\}$ and $R_G\{u_2, u_4\} =\{\ell'_1, u_2, u_4\}$, we have $g(\ell_1)+g(u_1)+g(u_3) \ge 1$ and $g(\ell'_1)+g(u_2)+g(u_4) \ge 1$; thus, $g(V)\ge 2$ and hence $\dim_f(G) \ge 2$. If $h:V\rightarrow [0,1]$ is a function defined by $h(\ell_1)=h(\ell'_1)=0$ and $h(v)=\frac{1}{2}$ for each $V- \{\ell_1, \ell'_1\}$, then $h$ forms a resolving function for both $G$ and $\overline{G}$ with $h(V)=2$; thus, $\max\{\dim_f(G),\dim_f(\overline{G}),\Sd_f(G, \overline{G})\}\ \le 2$. By Corollary~\ref{bounds_Gbar}, $\Sd_f(G, \overline{G})=\dim_f(G)=2=\frac{1}{3}+\dim_f(\overline{G})$.

Second, suppose $a\ge 2=1+b$ or $b\ge 2=1+a$, say the former. We show that $\Sd_f(G, \overline{G})=\dim_f(\overline{G})=\frac{a}{2}+\frac{4}{3}=\frac{1}{3}+\dim_f(G)$. Note that $g(\ell'_1)+g(u_2)+g(u_4)\ge 1$ since $R_G\{u_2, u_4\}=\{\ell'_1,u_2, u_4 \}$, and $g(\cup_{i=1}^{a}\{\ell_i\}) \ge \frac{a}{2}$ by Lemma~\ref{obs_twin_frac}; thus, $g(V)\ge \frac{a}{2}+1$ and hence $\dim_f(G)\ge\frac{a}{2}+1$. If we let $h':V\rightarrow [0,1]$ be a function defined by $h'(u_i)=0$ for each $i\in[3]$ and $h(v)=\frac{1}{2}$ for each $v\in V-\{u_1, u_2, u_3\}$, then $h'$ is a resolving function of $G$ with $h'(V)=\frac{a}{2}+1$; thus, $\dim_f(G)\le \frac{a}{2}+1$. Now, we note the following: (i) $\tilde{g}(\cup_{i=1}^{a}\{\ell_i\}) \ge \frac{a}{2}$ by Lemma~\ref{obs_twin_frac}; (ii) $R_{\overline{G}}\{u_2, u_4\}=\{\ell'_1, u_1, u_2, u_4\}$, $R_{\overline{G}}\{u_3, u_4\}=\{u_1, u_2,u_3, u_4\}$, $R_{\overline{G}}\{\ell'_1, u_2\}=\{\ell'_1,u_1,u_2, u_3\}$ and $R_{\overline{G}}\{\ell'_1, u_3\}=\{\ell'_1,u_3, u_4\}$, and thus we have $\tilde{g}(\ell'_1)+\tilde{g}(u_1)+\tilde{g}(u_2)+\tilde{g}(u_4)\ge 1$, $\tilde{g}(u_1)+\tilde{g}(u_2)+\tilde{g}(u_3)+\tilde{g}(u_4)\ge 1$, $\tilde{g}(\ell'_1)+\tilde{g}(u_1)+\tilde{g}(u_2)+\tilde{g}(u_3)\ge 1$ and $\tilde{g}(\ell'_1)+\tilde{g}(u_3)+\tilde{g}(u_4)\ge 1$. By summing over the four inequalities in (ii), we have $3\tilde{g}(V-(\cup_{i=1}^{a}\{\ell_i\}))\ge 4$, i.e, $\tilde{g}(V-(\cup_{i=1}^{a}\{\ell_i\}))\ge \frac{4}{3}$; this, together with (i), we have $\tilde{g}(V)\ge \frac{a}{2}+\frac{4}{3}$, and hence $\dim_f(\overline{G}) \ge \frac{a}{2}+\frac{4}{3}$. If $h:V \rightarrow [0,1]$ is a function defined by $h(u_1)=0$, $h(\ell_i)=\frac{1}{2}$ for each $i\in[a]$, and $h(v)=\frac{1}{3}$ for each $v\in V-(\{u_1\} \cup (\cup_{i=1}^{a}\{\ell_i\}))$, then $h$ is a resolving function for both $G$ and $\overline{G}$ with $h(V)=\frac{a}{2}+\frac{4}{3}$; thus, $\max\{\dim_f(G),\dim_f(\overline{G}),\Sd_f(G, \overline{G})\}\ \le \frac{a}{2}+\frac{4}{3}$. By Corollary~\ref{bounds_Gbar}, $\Sd_f(G, \overline{G})=\dim_f(\overline{G})=\frac{a}{2}+\frac{4}{3}=\frac{1}{3}+\dim_f(G)$.

Third, suppose $a \ge 2$ and $b \ge 2$. We show that $\dim_f(\overline{G})=\frac{a+b+2}{2}$. We note the following: (i) $\tilde{g}(\cup_{i=1}^{a}\{\ell_i\})\ge \frac{a}{2}$ and $\tilde{g}(\cup_{j=1}^{b}\{\ell'_j\})\ge \frac{b}{2}$ by Lemma~\ref{obs_twin_frac}; (ii) $\tilde{g}(\cup_{i=1}^{4}\{u_i\})\ge 1$ since $R_{\overline{G}}\{u_3, u_4\}=\{u_1, u_2, u_3, u_4\}$. So, $\tilde{g}(V)\ge \frac{a+b+2}{2}$, and thus $\dim_f(\overline{G})\ge \frac{a+b+2}{2}$. If $h:V\rightarrow [0,1]$ is a function defined by $h(u_1)=h(u_2)=0$ and $h(v)=\frac{1}{2}$ for each $v\in V-\{u_1, u_2\}$, then $h$ is a resolving function for both $G$ and $\overline{G}$ with $h(V)=\frac{a+b+2}{2}$; thus, $\max\{\dim_f(G),\dim_f(\overline{G}),\Sd_f(G, \overline{G})\}\le \frac{a+b+2}{2}$. By Corollary~\ref{bounds_Gbar}, $\Sd_f(G, \overline{G})= \dim_f(\overline{G})=\frac{a+b+2}{2}$.~\hfill
\end{proof}

Next, we consider $\Sd_f(G, \overline{G})$ for an arbitrary unicyclic graph $G$. We recall the following.

\begin{observation}\emph{\cite{SdimGbar}}\label{unicyclic_diam2}
If $G$ is a unicyclic graph of order $n\ge 4$ with $\diam(G)=2$, then $G\in\{C_4, C_5\}$ or $G$ is isomorphic to the graph obtained from the star $K_{1, n-1}$ by joining an edge between two end-vertices.
\end{observation}

\begin{lemma}\label{sdim_frac_unicyclic_diam2}
Let $G$ be a unicyclic graph of order $n\ge 4$ with $\diam(G)=2$. Let $H$ be the graph obtained from $K_{1, n-1}$ by joining an edge between two end-vertices. Then
\begin{equation*}
\Sd_f(G, \overline{G})=\dim_f(G)=\left\{
\begin{array}{ll}
\dim_f(\overline{G}) & \mbox{ if $G\in\{C_4,C_5\}$ or $G=H$ with $n\ge5$},\\
\dim_f(\overline{G})+\frac{1}{2} & \mbox{ if $G=H$ with $n=4$}.
\end{array}\right.
\end{equation*}
\end{lemma}

\begin{proof}
Let $G$ be a unicyclic graph of order $n\ge 4$ with $\diam(G)=2$. By Observation~\ref{unicyclic_diam2}, $G$ is isomorphic to $C_4$, $C_5$, or $H$. By Proposition~\ref{sdim_frac_cycle} and its proof, $\Sd_f(C_4, \overline{C}_4)=2=\dim_f(\overline{C}_4)=\dim_f(C_4)$ and $\Sd_f(C_5, \overline{C}_5)=\frac{5}{4}=\dim_f(\overline{C}_5)=\dim_f(C_5)$. Now, suppose $G$ is isomorphic to $H=K_{1, n-1}+e$ with $V(H)=\{u_0, u_1, \ldots, u_{n-1}\}$ such that $\deg_H(u_0)=n-1$ and $e=u_1u_2\in E(H)$. Let $g:V\rightarrow [0,1]$ be any resolving function for $H$ and let $\tilde{g}:V\rightarrow [0,1]$ be any resolving function for $\overline{H}$ on the common vertex set $V=V(H)=V(\overline{H})$. If $n\ge 5$, then $g(u_1)+g(u_2) \ge 1$ and $g(\cup_{i=3}^{n-1}\{u_i\}) \ge \frac{n-3}{2}$ by Lemma~\ref{obs_twin_frac}; thus, $g(V)\ge \frac{n-1}{2}$, and hence $\dim_f(H)\ge \frac{n-1}{2}$. The same argument works to show that $\tilde{g}(V)\ge \frac{n-1}{2}$ and $\dim_f(\overline{H})\ge \frac{n-1}{2}$. If $h:V \rightarrow [0,1]$ is a function defined by $h(u_0)=0$ and $h(u_i)=\frac{1}{2}$ for each $i\in[n-1]$, then $h$ forms a resolving function for both $H$ and $\overline{H}$; thus, $\max\{\dim_f(H), \dim_f(\overline{H}), \Sd_f(H, \overline{H})\} \le \frac{n-1}{2}$. By Corollary~\ref{bounds_Gbar}, $\Sd_f(H, \overline{H})= \dim_f(\overline{H})=\dim_f(H)=\frac{n-1}{2}$. So, suppose $n=4$; we show that $\Sd_f(H, \overline{H})=\frac{3}{2}=\dim_f(H)=\frac{1}{2}+\dim_f(\overline{H})$. Since $\overline{H}$ is isomorphic to the disjoint union of $K_1$ and $P_3$, $\dim_f(\overline{H})=1$ by Theorem~\ref{dim_frac_graph}(a). Since $R_G\{u_1, u_2\}=\{u_1, u_2\}$, $R_G\{u_0,u_1\}=\{u_0, u_1, u_3\}$ and $R_G\{u_0,u_2\}=\{u_0, u_2, u_3\}$, we have $g(u_1)+g(u_2)\ge 1$, $g(u_0)+g(u_1)+g(u_3)\ge 1$ and $g(u_0)+g(u_2)+g(u_3)\ge 1$. By summing over the three inequalities, we have $2g(V)\ge 3$, and hence $\dim_f(H) \ge \frac{3}{2}$. If $h':V\rightarrow[0,1]$ is a function defined by $h'(u_0)=0$ and $h'(u_i)=\frac{1}{2}$ for each $i\in[3]$, then $h'$ is a resolving function for both $H$ and $\overline{H}$; thus, $\max\{\dim_f(H), \dim_f(\overline{H}), \Sd_f(H, \overline{H})\} \le \frac{3}{2}$. By Corollary~\ref{bounds_Gbar}, $\Sd_f(H, \overline{H})=\dim_f(H)=\frac{3}{2}=\frac{1}{2}+\dim_f(\overline{H})$ for $n=4$.~\hfill
\end{proof}

\begin{theorem}
Let $G$ be a unicyclic graph of order at least three. Let $H_1$ be the graph obtained from $K_1 \cup C_3$ by joining an edge between the vertex in $K_1$ and a vertex in $C_3$, 
let $H_2$ be the graph obtained from $K_1 \cup C_4$ by joining an edge between the vertex in $K_1$ and a vertex in $C_4$, and let $H_3$ be the graph obtained from $P_6$ by joining an edge between the two support vertices of $P_6$. Then
\begin{equation*}
\Sd_f(G, \overline{G})=\left\{
\begin{array}{ll}
\dim_f(G)=\dim_f(\overline{G})+\frac{1}{2} & \mbox{ if } G\in\{H_1, H_2\},\\
\dim_f(G)=\dim_f(\overline{G})+\frac{1}{3} & \mbox{ if } G=H_3,\\
\dim_f(\overline{G}) & \mbox{ otherwise}.
\end{array}\right.
\end{equation*}
\end{theorem}

\begin{proof}
Let $G$ be a unicyclic graph of order at least three. We consider the following three cases: (i) $\diam(G)=1$; (ii) $\diam(G)=2$ or $\diam(\overline{G})=2$; (iii) $\diam(G)=3=\diam(\overline{G})$. 

If $\diam(G)=1$, then $G=C_3$ and $\Sd_f(C_3, \overline{C}_3)=\dim_f(\overline{C}_3)=\dim_f(C_3)=\frac{3}{2}$ by Proposition~\ref{sdim_frac_cycle}. If $\diam(G)=2$, then, by Observation~\ref{unicyclic_diam2}, $G\in\{C_4, C_5\}$ or $G$ is isomorphic to the graph $H_n^*$ obtained from the star $K_{1, n-1}$ by joining an edge between two end-vertices. By Lemma~\ref{sdim_frac_unicyclic_diam2}, we have the following: (i) if $G$ is isomorphic to $H_4^*=H_1$, then $\Sd_f(G, \overline{G})=\dim_f(G)=\dim_f(\overline{G})+\frac{1}{2}$; (ii) if $G\in\{C_4, C_5, H^*_n\}$ for $n\ge5$, then $\Sd_f(G, \overline{G})=\dim_f(\overline{G})$. If $\diam(\overline{G})=2$, then $\Sd_f(G, \overline{G})=\dim_f(\overline{G})$ by Theorem~\ref{thm_diam}. If $\diam(G)=3=\diam(\overline{G})$, then $G$ is isomorphic to one of the graphs in Figure~\ref{fig_unicyclic33} by Proposition~\ref{unicyclic_diam3}. By Proposition~\ref{sdim_frac_unicyclic_main}, we have the following: (i) if $G=H_2$, then $\Sd_f(G, \overline{G})=\dim_f(G)=\dim_f(\overline{G})+\frac{1}{2}$; (ii) if $G=H_3$, then $\Sd_f(G, \overline{G})=\dim_f(G)=\dim_f(\overline{G})+\frac{1}{3}$; (iii) if $G\not\in\{H_2, H_3\}$, then $\Sd_f(G, \overline{G})=\dim_f(\overline{G})$.~\hfill
\end{proof}

We conclude this section with a remark and an open problem.

\begin{remark}
There exists a graph $G$ of order $n\ge 2$ such that $\min\{\dim_f(G)+\dim_f(\overline{G}), \frac{n}{2}\}-\Sd_f(G, \overline{G})$ can be arbitrarily large. Let $v_1, v_2, \ldots, v_k$ be the exterior major vertices of a tree $G$ with $ter_G(v_i)=3$ for each $i\in[k]$, where $k\ge 2$ (see Figure~\ref{fig_ex1}). Let $g:V \rightarrow[0,1]$ be a resolving function for $G$ or $\overline{G}$ on the common vertex set $V=V(G)=V(\overline{G})$. By Lemma~\ref{obs_twin_frac}, $g(\{\ell_i, \ell'_i, \ell''_i\}) \ge \frac{3}{2}$ for each $i\in[k]$. Thus, $g(V) \ge \sum_{i=1}^{k}g(\{\ell_i, \ell'_i, \ell''_i\}) \ge \frac{3k}{2}$. So, $\dim_f(G)\ge \frac{3k}{2}$ and $\dim_f(\overline{G})\ge \frac{3k}{2}$. If we let $h:V \rightarrow [0,1]$ be a function defined by $h(v_i)=0$, for each $i\in[k]$, and $h(u)=\frac{1}{2}$ for each $u\in V-(\cup_{i=1}^{k}\{v_i\})$, then $h$ is a resolving function for both $G$ and $\overline{G}$ with $g(V)=\frac{3k}{2}$; thus, $\max\{\dim_f(G), \dim_f(\overline{G}), \Sd_f(G, \overline{G})\} \le \frac{3k}{2}$. So, $\Sd_f(G, \overline{G})=\dim_f(G)=\dim_f(\overline{G})=\frac{3k}{2}$. Therefore, $\min\{\dim_f(G)+\dim_f(\overline{G}), \frac{n}{2}\}-\Sd_f(G, \overline{G})=2k-\frac{3k}{2}=\frac{k}{2} \rightarrow \infty$ as $k\rightarrow \infty$.
\end{remark}

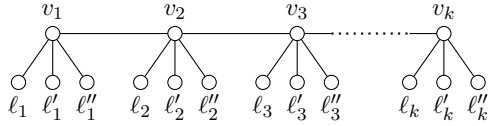
\begin{figure}[ht]
\centering
\begin{tikzpicture}[scale=.65, transform shape]

\node [draw, shape=circle, scale=.8] (v1) at  (0,0) {};
\node [draw, shape=circle, scale=.8] (v2) at  (2.5,0) {};
\node [draw, shape=circle, scale=.8] (v3) at  (5,0) {};
\node [draw, shape=circle, scale=.8] (v4) at  (8,0) {};
\node [draw, shape=circle, scale=.8] (11) at  (-0.7,-1) {};
\node [draw, shape=circle, scale=.8] (12) at  (0,-1) {};
\node [draw, shape=circle, scale=.8] (13) at  (0.7,-1) {};
\node [draw, shape=circle, scale=.8] (21) at  (1.8,-1) {};
\node [draw, shape=circle, scale=.8] (22) at  (2.5,-1) {};
\node [draw, shape=circle, scale=.8] (23) at  (3.2,-1) {};
\node [draw, shape=circle, scale=.8] (31) at  (4.3,-1) {};
\node [draw, shape=circle, scale=.8] (32) at  (5,-1) {};
\node [draw, shape=circle, scale=.8] (33) at  (5.7,-1) {};

\node [draw, shape=circle, scale=.8] (41) at  (7.3,-1) {};
\node [draw, shape=circle, scale=.8] (42) at  (8,-1) {};
\node [draw, shape=circle, scale=.8] (43) at  (8.7,-1) {};

\node [scale=1.3] at (0,0.45) {$v_1$};
\node [scale=1.3] at (2.5,0.45) {$v_2$};
\node [scale=1.3] at (5,0.45) {$v_3$};
\node [scale=1.3] at (8,0.45) {$v_{k}$};
\node [scale=1.3] at (-0.7,-1.5) {$\ell_{1}$};
\node [scale=1.3] at (0,-1.5) {$\ell'_{1}$};
\node [scale=1.3] at (0.7,-1.5) {$\ell''_{1}$};
\node [scale=1.3] at (1.8,-1.5) {$\ell_2$};
\node [scale=1.3] at (2.5,-1.5) {$\ell'_2$};
\node [scale=1.3] at (3.2,-1.5) {$\ell''_2$};
\node [scale=1.3] at (4.3,-1.5) {$\ell_3$};
\node [scale=1.3] at (5,-1.5) {$\ell'_3$};
\node [scale=1.3] at (5.7,-1.5) {$\ell''_3$};

\node [scale=1.3] at (7.3,-1.5) {$\ell_{k}$};
\node [scale=1.3] at (8,-1.5) {$\ell'_{k}$};
\node [scale=1.3] at (8.7,-1.5) {$\ell''_{k}$};

\draw(v1)--(v2)--(v3)--(5.7,0); \draw[thick, dotted](5.7,0)--(7.3,0);\draw(7.3,0)--(v4);
\draw(11)--(v1)--(12);\draw(21)--(v2)--(22);\draw(31)--(v3)--(32);\draw(41)--(v4)--(42);
\draw(v1)--(13);\draw(v2)--(23);\draw(v3)--(33);\draw(v4)--(43);

\end{tikzpicture}
\caption{\small{A graph $G$ such that $\min\{\dim_f(G)+\dim_f(\overline{G}), \frac{|V(G)|}{2}\}-\Sd_f(G, \overline{G})$ can be arbitrarily large, where $k\ge 2$.}}
\label{fig_ex1}
\end{figure}

\begin{question}
We note that $\Sd_f(P_4, \overline{P}_4)=\frac{4}{3}>1=\max\{\dim_f(P_4), \dim_f(\overline{P}_4)\}$ by Proposition~\ref{sim_fdim_path} and Theorem~\ref{dim_frac_graph}(a). Can $\Sd_f(G, \overline{G})-\max\{\dim_f(G), \dim_f(\overline{G})\}$ be arbitrarily large as $G$ varies?
\end{question}

\medskip

\textbf{Acknowledgement.} This research was partially supported by US-Slovenia Bilateral Collaboration Grant (BI-US/19-21-077). The authors thank the anonymous referees for some helpful comments which improved the presentation of the paper.

%%%%%%%%%%%%%%%%%%%%%%%%%
%%%%%%%%%%%%%%%%%%%%%%%%%

\end{document}